\documentclass[aos]{imsart}

\RequirePackage[OT1]{fontenc}
\usepackage{amsmath,amsthm,amssymb}
\usepackage{amsfonts,dsfont}
\usepackage{pstricks}
\usepackage{comment}
\usepackage{natbib}
\usepackage{colordvi}
\usepackage{graphicx}
\usepackage{rotating}

\arxiv{1002.2845}

\startlocaldefs
\usepackage{amsthm}
\newtheorem{theorem}{Theorem}[section]
\newtheorem{proposition}[theorem]{Proposition}
\newtheorem{corollary}[theorem]{Corollary}
\newtheorem{lemma}[theorem]{Lemma}

\newtheorem{remark}[theorem]{Remark}

\newtheorem{definition}[theorem]{Definition}
\newcommand{\R}{\mathbb{R}} % set of real numbers
\renewcommand{\P}{\mathbb{P}}
\newcommand{\E}{\mathbb{E}} % esperance
\newcommand{\V}{\mathbb{V}} % esperance
\newcommand{\telque}{\;|\;}
\newcommand{\FDR}{\mbox{FDR}}
\newcommand{\pFDR}{\mbox{pFDR}}

\newcommand{\mbp}{\mathbb{P}}

\newcommand{\mtc}{\mathcal}
\newcommand{\mbf}{\mathbf}

\newcommand{\wh}[1]{{\widehat{#1}}}
\newcommand{\ol}[1]{\overline{#1}}

\newcommand{\ind}[1]{{\mbf{1}\{#1\}}}

\newcommand{\prob}[1]{\mbp\brac{#1}}

\newcommand{\brac}[1]{\left[#1\right]}

\newcommand{\cH}{{\mtc{H}}}

\newcommand{\Pow}{\mbox{Pow}}
\newcommand{\LSU}{\mbox{LSU}}
\newcommand{\LSD}{\mbox{LSD}}

\newcommand{\SU}{\mbox{SU}}
\newcommand{\SD}{\mbox{SD}}

\newcommand{\distsu}{\mathcal{D}}
\newcommand{\distsd}{\widetilde{\mathcal{D}}}
\usepackage{graphics}
\newcommand{\FDP}{\mbox{FDP}} 
\newcommand{\FNR}{\mbox{FNR}} 
\newcommand{\minrho}{-(m-1)^{-1}} 
\renewcommand{\H}{H} 
\newcommand{\PExI}{\ol{P}^{I}}
\newcommand{\PExN}{\ol{P}^{N}}
\renewcommand{\l}{\ell}

\newcommand{\Sti}[2]{\genfrac{\{}{\}}{0pt}{}{#1}{#2}}
\endlocaldefs

\begin{document}

\begin{frontmatter}

% "Title of the paper"

%\title{Exact calculation of false discovery rate for step-up and step-down procedures} %under independence and equicorrelated Gaussian dependence
%\title{Exact calculation of false discovery rate with application to least favorable configurations}
\title{Exact calculations for false discovery proportion with application to least favorable configurations}
\runtitle{Exact calculation for FDP}

\runauthor{Roquain, E. and Villers, F.}

\begin{aug}
\author{\fnms{Etienne} \snm{Roquain}\ead[label=e1]{etienne.roquain@upmc.fr}}
\address{
UPMC Univ Paris 6, LPMA,\\ 4, Place Jussieu, 75252 Paris cedex 05, France\\
\printead{e1}\\
\phantom{E-mail: etienne.roquain@upmc.fr\ }}
\and
\author{\fnms{Fanny} \snm{Villers}\ead[label=e2]{fanny.villers@upmc.fr}}
\address{
UPMC Univ Paris 6, LPMA,\\ 4, Place Jussieu, 75252 Paris cedex 05, France\\
\printead{e2}\\
\phantom{}}

\runauthor{Roquain, E. and Villers, F.}
\affiliation{UPMC University of Paris 6}
\end{aug}

\begin{abstract}
In a context of multiple hypothesis testing, we provide several new exact calculations related to the false discovery proportion (FDP) of step-up and step-down procedures. For step-up procedures, we show that the number of erroneous rejections conditionally on the rejection number is simply a binomial variable, which leads to  explicit  computations of the c.d.f., the {$s$-th} moment and the mean of the FDP, the latter corresponding to the false discovery rate (FDR). For step-down procedures, we derive what is to our knowledge the first explicit formula for the FDR valid for any alternative c.d.f.  of the $p$-values. We also derive explicit computations of the power for both step-up and step-down procedures.
These formulas are ``explicit'' in the sense that they only involve the parameters of the model and the c.d.f. of the order statistics of i.i.d. uniform variables. The $p$-values are assumed either independent or coming from an equicorrelated multivariate normal model and an additional mixture model for the true/false hypotheses is used. 
This new approach is used to investigate new results  which are of interest in their own right,  related to least/most favorable configurations for the FDR and  the variance of  the FDP.
%We propose several applications of theses formulas. In particular, some new results related to least/most favorable configurations for the FDR and the variance of  the FDP are presented.
\end{abstract}



\end{frontmatter}

% indicate corresponding author with \corref{}
% \author{\fnms{John} \snm{Smith}\thanksref{t2}\corref{}\ead[label=e1]{smith@foo.com}\ead[label=e2,url]{www.foo.com}}
% \thankstext{t2}{Thanks to somebody}
% \address{line 1\\ line 2\\ \printead{e1}\\ \printead{e2}}

%\tableofcontents

\section{Introduction}\label{sec:intro}

When testing simultaneously $m$ null hypotheses, the false discovery proportion (FDP) is defined as the proportion of errors among all the rejected hypotheses and the false discovery rate (FDR) is defined as the average of the FDP. Since its introduction by Benjamini and Hochberg (1995) \cite{BH1995}, the FDR has become a widely used type I error criterion, because it is adaptive to the number of rejected hypotheses. However, the randomness of the denominator in the FDP expression makes the study of the FDP and of the FDR mathematically challenging.

There is a considerable number of papers that deal with the FDR control under different dependency conditions between the $p$-values (see for instance \cite{BH1995,BY2001, Sar2008, BR2008EJS, BR2008b}). In the latter, the goal is, given a prespecified level $\alpha$, to provide a procedure with a FDR smaller than $\alpha$ (for any value of the data law in a given distribution subspace, e.g. for some dependency assumptions). For instance, the famous linear step-up procedure (LSU), also called the Benjamini-Hochberg procedure \cite{BH1995} (based on the Simes's line \cite{Sim1986}), has been proved to control the FDR under independence and under a positive dependence assumption (see \cite{BH1995,BY2001}). While controlling the FDR, one wants to maximize the power of the procedure, the power being generally defined as the  averaged number of correct rejections divided by the number of false hypotheses.

In this paper, we deal with the ``reversed'' approach: given the procedure, we aim to compute the corresponding FDR,  or more generally the $s$-th moment and the c.d.f. of the FDP, and the Power. 
For procedures using a constant thresholding and under a mixture model assuming independence between the $p$-values, Storey (2003) \cite{Storey2003}  addressed theses questions, while introducing the positive false discovery
rate (pFDR) (see also Chi and Tan \cite{Chi2008}). 
Considering step-up or step-down methods requires more efforts: when the $p$-values are all i.i.d. uniform, 
the exact distribution of the rejection number has been computed by Finner and Roters (2002)
 for step-up and step-down procedures, which leads to a computation of the FDR in the degenerate case where all the hypotheses are true \cite{FR2002}.
%When all the
%hypotheses are true and the $p$-values are
%independent, the exact distribution of the rejection number has been computed in \cite{FR2002}
 %for step-up and step-down procedures. 
 When the $p$-values follow the particular ``Dirac-uniform'' configuration, that is, when the $p$-values associated to false hypotheses are equal to $0$ and when the $p$-values associated to true hypotheses are i.i.d. uniform, the FDP distribution has been computed by Dickhaus (2008) for general step-up-down procedures (see Section~3.7 in \cite{Dick2008}). %, including the step-up and the step-down cases. 
% As noticed in  \cite{Dick2008}, this straightforwardly leads to an exact computation of the FDR in this particular Dirac-uniform configuration.
 For an arbitrary distribution of the $p$-values under the alternative, Ferreira and Zwinderman (2006) gave a first exact expression for the moment of the FDP of the LSU procedure under independence  \cite{FZ2006}. 
%For other motivations, \cite{Sar2002} and \cite{GR2008} also provided exact expressions for the FDR and 
Together with other recent approaches  (see, e.g. \cite{Roq2007, Sar2008, BR2008EJS, RW2009}), this puts forward a connection between the FDR expression and the rejection number distribution in the step-up case, under independence of the $p$-values.
Additionally, Sarkar (2002) found an exact formula for the FDR, which is valid for any step-up-down procedure  \cite{Sar2002}. However, it  involves the c.d.f. of ordered components of dependent variables, by contrast with \cite{FR2002,Dick2008} and present paper, using c.d.f. of ordered components of i.i.d. uniform variables, so obtaining substantially more explicit formulas.

Meanwhile, some of these approaches have also been investigated from the asymptotic point of view, when the number of hypotheses $m$ tends to infinity;  Chi (2007)  computed the asymptotic rejection number distribution of the LSU procedure by introducing the criticality phenomenon \cite{Chi2007}, while Finner et al. (2007)  computed the limiting FDR of the LSU procedure for positively correlated $p$-values (following an equicorrelated multivariate normal model) in the particular Dirac-uniform configuration \cite{FDR2007}. In this paper, the point of view will be mainly nonasymptotic. 
%Next, a calculation of the power of the LSU have been advocated in \cite{Glueck2008}, but their formulas are reported to be very extremely cumbersome when the hypotheses number $m$ grows. 

The new contributions of the present paper are as follows:
\begin{itemize}
\item[\textbullet] For a step-up procedure using a threshold $(t_k)_k$, we proved that the distribution of the number of erroneous rejections conditionally on $k$ rejections is a binomial variable of parameters $k$ and $\pFDR(t_k)={\pi_0 t_k}/{G(t_k)}$, where $\pi_0$ is the (averaged) proportion of true nulls and $G$ is the c.d.f. of the $p$-values. This  provides new explicit formulas for  the c.d.f. of the FDP, the $s$-th moment of the FDP (providing a correction with respect to \cite{FZ2006} for $s\geq 3$) and for the FDR, for any alternative distribution of the $p$-values.
We also give an expression for the power, which yields a considerably less complex calculation than in \cite{Glueck2008}, see Section~\ref{sec:mainresults}.
\item[\textbullet] Considering a step-down procedure, a new  explicit formula for the FDR and the power  is presented under any alternative distribution of the $p$-values. To our knowledge, this expression is the first one  that clearly relates the FDR  to (joint) rejection number distribution in the case of a step-down procedure and that is valid under any alternative, see Section~\ref{sec:mainresults}.
\item[\textbullet]  All the previous results, valid under independence between the $p$-values, 
 can be easily extended to the case where the $p$-value family follows an equicorrelated multivariate normal model, by using a simple modification, see Section~\ref{sec:deppossu}. However, this requires the use of a nonnegative correlation. The case of a possibly negative correlation is considered when only two hypotheses are tested, see Section~\ref{sec:depneg}.
%jensuislà
\item[\textbullet] Our formulas corroborate the classical multiple testing results while they give rise to several new results. The two main corollaries hold under independence  and are as follows. First, in Section~\ref{appli_indep},  for the linear step-down procedure,  we prove that a $p$-value configuration maximizing the FDR, i.e. a least favorable configuration (LFC) for the FDR, is the Dirac-uniform configuration.
 Additionally, considering a general step-down procedure, we define a new condition on the threshold ensuring that the Dirac-uniform configuration is still a LFC. As discussed in Section~\ref{appli_indep}, this condition is different from the one of the step-up case. Second, 
 we found an exact expression of the minimum and the maximum of the variance of the FDP of the LSU, these extrema being taken over some $p$-value configuration sets of interest. The latter allows to better understand the behavior of the FDP around the FDR. % when $m$ tends to infinity. 
 In particular, this puts forward that the convergence of the FDP towards the FDR is particularly slow in the sparse case,
% a perhaps surprising behavior in the sparse case where $\pi_0$ tends to $1$
see Section~\ref{sec:LFCvar}.
\end{itemize}

 All our formulas are valid nonasymptotically, that is, they hold for each $m\geq 2$. As a counterpart, they inevitably have a general form that can appear somewhat complex at first sight. For instance, denoting by $\Psi_m$ the c.d.f. of the order statistics of $m$ i.i.d. uniform variables on $(0,1)$, the FDR formula for step-up procedures requires the computation of $\Psi_m$  at a given point of $\mathbb{R}^m$ (at most), while the FDR formula for step-down procedures requires the computation of $\Psi_m$ at $2m$ different points of $\mathbb{R}^m$ (at most). 
However, 
%Finally, in a nutshell, 
let us underline what are to our opinion the two main interests of this exact approach: % (see Section~\ref{appliLFC} for illustrations):
\begin{itemize}
\item[-]  For some model parameter configurations and after possible simplifications, the formulas are usable for further theoretical studies (monotonicity with respect to a parameter, convergence when $m$ tends to infinity,...), as in Theorem~\ref{appli_indep_sd} and Theorem~\ref{cor_LCPVAR}.
\item[-] %For $m$ not too large (say $m\leq 1000$), these formulas avoid using cumbersome and less accurate simulation experiments (the latter are extensively used in multiple testing literature). 
For $m$ not too large (say $m\leq 1000$), these formulas can be fully computed numerically, e.g. plotting exact graphs. Thus, they
avoid using cumbersome and less accurate simulation experiments (extensively used in multiple testing literature), see for instance Section~\ref{appli_dep}.
\end{itemize}

%Finally, although the model parameters are always assumed to be known in the paper, remark that they can be replaced by convergent estimators in the formulas.

\section{Preliminaries}\label{sec:prel}

\subsection{Models for the $p$-value family}\label{sec:setting}

On a given probability space, we consider a finite set of $m\geq 2$ null hypotheses, tested by a family of $m$ $p$-values $\mbf{p}=(p_i, i\in\{1,...,m\})$. 
In this paper, for simplicity, we skip somewhat the formal definition of $p$-values by defining directly a $p$-value model, that is, by specifying the joint distribution of $\mbf{p}$.

In what follows, we denote by $\mathcal{F}$ the set containing c.d.f.'s from $[0,1]$ into $[0,1]$ that are continuous 
and by $F_0(t)=t$ the c.d.f. of the uniform distribution over $[0,1]$ (we restricted our attention to the case where $F_0(t)=t$ only for simplicity, all our formulas will be valid for an arbitrary $F_0\in \mathcal{F}$).

\begin{definition}[Conditional $p$-value models] 
\begin{itemize}
\item[\textbullet]
The $p$-value family $\mbf{p}$ follows the \textit{conditional independent model} with parameters $\H=(H_i)_{1\leq i\leq m}\in \{0,1\}^m$ and $F_1\in \mathcal{F}$, that we denote by  $\mbf{p}\sim P^I_{(H,F_1)} $, if $\mbf{p}=(p_i, i\in\{1,...,m\})$ is a family of mutually independent variables and for all $i$, 
$$p_i \sim \left\{\begin{array}{ll} F_0 & \mbox{ if } H_i=0 \\  F_1 &\mbox{ if } H_i=1\end{array}\right..$$
\item[\textbullet]
The $p$-value family $\mbf{p}$ follows the \textit{conditional equicorrelated multivariate normal model} (say for short, conditional EMN model) with parameters $\H=(H_i)_{1\leq i\leq m}\in \{0,1\}^m$, $ \rho\in[\minrho,1],$ and $\mu>0$, that we denote by $\mbf{p}\sim P^N_{(\H,\rho,\mu)}$,  if for all $i,$ $p_i\sim \overline{\Phi}(X_i+\mu H_i )$,
  where the vector $(X_i)_{1\leq i \leq m}$ is distributed as a $\mathbb{R}^m$-valued Gaussian vector with zero means and a covariance matrix having $1$ on the diagonal and $\rho$ elsewhere and where $\overline{\Phi}$ denotes the standard Gaussian distribution tail, that is, $\overline{\Phi}(z)=\prob{Z\geq z}$ for $Z\sim\mathcal{N}(0,1)$.
 In that model, the marginal distributions of the $p$-values are thus given by: 
$$p_i \sim \left\{\begin{array}{ll} F_0& \mbox{ if } H_i=0 \\  F_1(t)=\overline{\Phi}\big(\overline{\Phi}^{-1}(t)-\mu\big) &\mbox{ if } H_i=1\end{array}\right..$$
\end{itemize}
\end{definition}

The two above models are said ``conditional'' because the distribution of the $p$-values are defined conditionally on the value of  the parameter $ \H=(H_i)_{1\leq i\leq m} \in\{0,1\}^m$. The latter determines which hypotheses are true or false: $H_i =0$ (resp. $1$) if the $i$-th null hypothesis is true (resp. false). We then denote by $\cH_0(\H):=\{i\in\{1,...,m\} \telque H_i=0\}$ the set corresponding to the true null hypotheses and by $m_0(H):=|\cH_0(H)|$ its cardinal. Analogously, we define $\cH_1(H):=\{i\in\{1,...,m\} \telque H_i=1\}$ and $m_1(H):=|\cH_1(H)|=m-m_0(H)$. 

To each one of the above models, we associate the ``unconditional" version %corresponding to the setting 
in which we endow the parameter $\H$ with the prior distribution $\mathcal{B}(1-\pi_0)^{\otimes m}$, making $(H_i)_{1\leq i\leq m}\in\{0,1\}^m$ being a sequence of i.i.d. Bernoulli with parameter $1-\pi_0$. On an intuitive point of view, this means that each hypothesis is true with probability $\pi_0$, independently from the other hypotheses. We thus define the following models for $\mbf{p}$ (or more precisely for $(H,\mbf{p})$): 

\begin{definition}[Unconditional $p$-value models] 
\begin{itemize}
\item[\textbullet] The couple $(H,\mbf{p})$ follows the unconditional independent model with parameters $\pi_0\in [0,1]$ and $F_1\in\mathcal{F}$, that we denote by $(H,\mbf{p}) \sim \PExI_{(\pi_0,F_1)}$ if $H\sim \mathcal{B}(1-\pi_0)^{\otimes m}$ and 
 the distribution of $\mbf{p}$ conditionally to $ \H$ is $P^I_{(\H,F_1)}$, that is, conditionally on $\H$, $\mbf{p}$ follows the conditional independent model with parameters $\H$ and $F_1$. 
In that model, the $p$-values are i.i.d. with common c.d.f. $G(t)=\pi_0 F_0(t) +(1-\pi_0) F_1(t)$. 
\item[\textbullet] The couple $(H,\mbf{p})$ follows the unconditional equicorrelated multivariate normal model (say for short, unconditional EMN model) with parameters $\pi_0\in [0,1], \rho\in[\minrho,1],$ and $ \mu>0$, that we denote by $(H,\mbf{p})\sim \PExN_{(\pi_0,\rho,\mu)} $, if $H\sim \mathcal{B}(1-\pi_0)^{\otimes m}$ and the distribution of $\mbf{p}$ conditionally on $\H$ is $P^N_{(\H,\rho,\mu)}$, that is, conditionally on $\H$, $\mbf{p}$ follows the conditional EMN model with parameters $\H$, $\rho$ and $\mu$. 
\end{itemize}
\end{definition}

An important point is that the quantities $m_0(H)$ and $m_1(H)$ are deterministic in the conditional models $P^I, P^N$, while they become random in the unconditional models $\PExI, \PExN$ with $m_0(H) \sim\mathcal{B}(m,\pi_0)$ and $m_1(H) \sim\mathcal{B}(m,1-\pi_0)$.

The conditional independent model is one of the most standard $p$-value models and was for instance considered in the original paper of Benjamini and Hochberg (1995) \cite{BH1995}. Its unconditional version, also called the ``random effects model'', is very convenient and has been widely used since its introduction by Efron et al. (2001) \cite{ETST2001}, see for instance \cite{Storey2003,GW2004}.

The conditional EMN model is a simple instance of model introducing dependencies between the $p$-values.  It corresponds to a one-sided testing on the mean of $X_i+\mu H_i$, simultaneously for all $1\leq i \leq m$. It
has become quite standard in recent FDR multiple testing literature; for instance, it was used in \cite{FDR2007} with $\mu=\infty$ and it has been considered in \cite{BKY2006,BR2008b} for numerical experiments. 
Furthermore,  Efron (2009)  recently showed that the EMN model may also be viewed as an approximation for some non-equicorrelated models, which reinforces its interest for a practical use  \cite{Efron2009}. % so that it also makes sense from a practical point of view.
%as a natural framework to test the FDR robustness over a particular class of procedures. 
 In this model, provided that $\rho\geq 0$,   the $p$-values are positively regression dependent on each one on the subset $\cH_0(H)$ (PRDS on  $\cH_0(H)$) which is one dependency condition that suffices for FDR control  (see \citep{BY2001}). The unconditional version of this model is convenient because it provides exchangeable $p$-values (although not independent when $\rho\neq 0$).
 
 Additionally, we will sometimes consider the ``Dirac-uniform configuration'' for the above models.  In that configuration, all the $p$-values corresponding to false nulls ($H_i=1$) are equal to  zero, that is, $F_1$ is constantly equal  to $1$ for the independent models and $\mu=\infty$ for the EMN models. 
This configuration was introduced in \cite{FDR2007} to increase the FDR as much as possible  for the linear step-up procedures which thus appears as a ``least favorable configuration'' for the FDR (see also Section~\ref{sec:LFCFDR}).
%Theorem~5.3 in \cite{BY2001}).  

  \subsection{Multiple testing procedures, FDP, FDR and power}

A \textit{multiple testing procedure} $R$ is defined as an algorithm which, from the data, aims to reject part of the null hypotheses.
Below, we will consider, as is usually the case, multiple testing procedures which can be written as a function of the $p$-value family $\mathbf{p}=(p_i, i\in\{1,...,m\})$.
More formally, we define a multiple testing procedure as a measurable function $R$, which takes as input a realization of the $p$-value family $\mbf{p}\in [0,1]^m$ and which returns a subset $R(\mbf{p})$ of $\{1,...,m\}$, corresponding to the rejected hypotheses (i.e.  $i\in R(\mbf{p})$ means that the $i$-th hypothesis is rejected by $R$ for the observed $p$-values $\mbf{p}$).

Particular multiple testing procedures are step-up and step-down procedures. 
First define a \textit{threshold} as any nondecreasing sequence $\mbf{t}= (t_k)_{1\leq k \leq m} \in [0,1]^m$ (with $t_0=0$ by convention).
Next, for any threshold $\mbf{t}$, the \textit{step-up procedure} of threshold $\mbf{t}$, denoted here by {\SU}$(\mbf{t})$,  rejects the $i$-th hypothesis if $p_i\leq t_{\hat{k}}$, with 
% \begin{equation}
$\wh{k}= \max\{k\in\{0,1,...,m\}\telque p_{(k)}\leq {t}_k\},$
%\label{def_SU_bad}\end{equation}
 where $p_{(1)}\leq p_{(2)} \leq ...\leq p_{(m)}$ denote the ordered $p$-values (with the convention $p_{(0)}=0$).
In particular, the procedure {\SU}$(\mbf{t})$ using $t_k=\alpha k/m$ corresponds to the standard linear step-up procedure of Benjamini and Hochberg (1995) \cite{BH1995}, denoted here by $\LSU$. % (linear step-up procedure).
A less rejecting procedure uses a step-down algorithm; for any threshold $\mbf{t}$, the \textit{step-down procedure} of threshold $\mbf{t}$, denoted here by {\SD}$(\mbf{t})$,  rejects the $i$-th hypothesis if $p_i\leq t_{\tilde{k}}$, with 
% \begin{equation}
$\widetilde{k}= \max\{k\in\{0,1,...,m\}\telque \forall k'\leq k, \:p_{(k')}\leq {t}_{k'}\}.$
%\label{def_SD_bad}\end{equation}
%\medskip
Analogously to the step-up case,  the procedure {\SD}$(\mbf{t})$ using $t_k=\alpha k/m$ is called the linear step-down procedure and is denoted here by $\LSD$. 

Next, associated to  any multiple testing procedure $R$ and any configuration of true/false hypotheses $H\in\{0,1\}^m$, we introduce the false  discovery proportion (FDP) of $R$ as the proportion of true hypotheses in the set of the rejected hypotheses, 
that is,
\begin{eqnarray}
\label{equ_FDP}
\FDP(R,H)&=& \frac{|\cH_0(H)\cap R|}{|R|\vee 1}\,,
\end{eqnarray}
where $|\cdot|$ denotes the cardinality function. Then,  while for any multiple testing procedure $R$, the false discovery rate  (FDR) is defined as the mean of the FDP (see \citep{BH1995}), the power is (generally) defined as the expected number of correctly rejected hypotheses divided by the number of false hypotheses. 
Of course, the FDR and the power depend on the distribution that generates the $p$-values, and we may use the models defined in  Section~\ref{sec:setting}. Formally, for any distribution $P$ coming from a conditional model using parameter $H\in\{0,1\}^m$, we let 
\begin{eqnarray}
\label{equ_FDR}
\FDR(R,P)&=&\E_{\mathbf{p}\sim P} [\FDP(R(\mbf{p}),H)],\\
\label{equ_Pow}
\Pow(R,P)&=&m_1(H)^{-1}\: \E_{\mathbf{p}\sim P}\big[ |\cH_1(H)\cap R(\mbf{p})|\big].
\end{eqnarray}
Similarly, for any $p$-value distribution $\ol{P}$ coming from an unconditional model, the FDR and the Power use an additional averaging over $H\sim \mathcal{B}(1-\pi_0)^{\otimes m}$ and are defined by: 
\begin{eqnarray}
\FDR(R,\ol{P})&=&\E_{(H,\mathbf{p})\sim \ol{P}}[\FDP(R(\mbf{p}),H) ]\label{equ_FDR_ex},\\
\Pow(R,\ol{P})&=&(\pi_1m)^{-1} \:\E_{(H,\mathbf{p})\sim \ol{P}}\big[ |\cH_1(H)\cap R(\mbf{p})|\big].\label{equ_Pow_ex}
\end{eqnarray}
Remark that, for convenience, \eqref{equ_Pow_ex} is not exactly defined as the expectation of \eqref{equ_Pow}, because of the denominator. It corresponds precisely to the expected number of correctly rejected hypotheses divided by the \textit{expected} number of false hypotheses.
% Of course, the FDR in \eqref{equ_FDR} depends on the model chosen for the $p$-values. In particular, the FDR in the conditional model involves an expectation taken conditionally on $\H$, while the FDR in the unconditional model additionally uses an averaging over $\H$.
%As a consequence, the FDR in \eqref{equ_FDR_ex} use an additional averaging with respect to FDR in \eqref{equ_FDR}.

In the paper, to simplify the notation, we sometimes drop the explicit dependency in $\mbf{p}$, $H$ or $P$, writing e.g. $R$ instead of $R(\mbf{p})$, $\cH_0$ instead of $\cH_0(H)$,  $\FDP(R)$ instead of $\FDP(R,H)$ and $\FDR(R)$ instead of $\FDR(R,P)$.

\subsection{Some notation and useful results}\label{sec:notation}

For any $k\geq 0$ and any threshold $\mbf{t}= (t_1,...,t_k)$, we denote
\begin{equation}
\Psi_k(\mbf{t}) = \Psi_k(t_{1},...,t_{k}) =\prob{U_{(1)}\leq t_{1}, ..., U_{(k)}\leq t_k}.\label{equ_psi}
\end{equation}
 where $(U_i)_{1\leq i\leq k}$ is a sequence of $k$ variables i.i.d. uniform on $(0,1)$ and
with the convention $\Psi_0(\cdot)=1$.
In practice, quantity \eqref{equ_psi} can be evaluated using Bolshev's recursion
%\begin{align*}
%$\Psi_k(\mbf{t}) = \det\bigg( \bigg[\frac{j(t_j)^{j-i+1}}{(j-i+1)!}\ind{j-i+1\geq 0}\bigg]_{1\leq i,j\leq k} \bigg),$%\\
$\Psi_k(\mbf{t})  = 1- \sum_{i=1}^k { {k}\choose{i}} (1-t_{k-i+1})^i \Psi_{k-i}(t_1,...,t_{k-i}) $ %\\ 
or Steck's recursion
$\Psi_k(\mbf{t})  = (t_k)^k- \sum_{j=0}^{k-2} { {k}\choose{j}} (t_k-t_{j+1})^{k-j} \Psi_{j}(t_1,...,t_j)$
%\end{align*}
  (see \citep{SW1986}, p. 366-369).
Additionally, the following relation holds (see Lemma 2.1  in \cite{FR2002}): for all $k\in\mathbb{N}$ and $\nu_1,\nu_2\in \R$ such that $0\leq \nu_1+\nu_2\leq  \nu_1+k\nu_2\leq 1,$
\begin{equation}\label{relFinner}
\Psi_k( \nu_1+\nu_2,..., \nu_1+k\nu_2)=( \nu_1+\nu_2)(\nu_1+(k+1)\nu_2)^{k-1}.
\end{equation}

From the $\Psi_k$'s, we define the following useful quantities: for any threshold $\mbf{t}=(t_k)_{1\leq k \leq m}$ and $k\geq 0$, $k\leq m$, we let
\begin{align}
\distsu_m(\mbf{t},k) &=  {m \choose k} (t_k)^k \Psi_{m-k}\big( 1-t_m,...,1-t_{k+1}\big)\label{equ_for_distsu},\\
\distsd_m(\mbf{t},k) &= {m \choose k} (1-t_{k+1})^{m-k} \Psi_{k}\big( t_{1},...,t_{k}\big).\label{equ_for_distsd}
\end{align}
Above, note that $(t_k)^k$ and $(1-t_{k+1})^{m-k}$ are correct when $k=0$ and $k=m$, even if $(t_j)_j$ is only defined for $1\leq j\leq m$. % expression \eqref{equ_for_distsu} (resp. \eqref{equ_for_distsd}) above is well defined for $k=0$ (resp. $k=m$) because it does not depend on the value $t_0$ (resp. on the value $t_{m+1}$). 
Note that Bolshev's recursion provides $\sum_{k=0}^m\distsu_m(\mbf{t},k)=\sum_{k=0}^m\distsd_m(\mbf{t},k)=1$ for any threshold $\mbf{t}$. 

%, by using the convention $00=1$.

%Next, define for any threshold $\mbf{t}=(t_j)_{1\leq j \leq m+1}$, and $k,k'\geq 1$, $k\leq k'\leq m$,
%\begin{align}
%\distsd_m(\mbf{t},k,k') =&\: \distsd_m((t_j)_{1\leq j\leq m},k)\: \distsd_{m-k} \left( \left(\frac{t_{k+1+j}-t_{k+1}}{1-t_{k+1}}\right)_{1\leq j\leq m-k},\:k'-k \right) \label{equ_astuce}\\
%=&  {m \choose k} {m-k \choose k'-k}   (1-t_{k'+2})^{m-k'} % \nonumber\\& \times 
%\Psi_{k}\big( t_1,...,t_k\big)
  %\Psi_{k'-k}\big( t_{k+2}-t_{k+1},...,t_{k'+1}-t_{k+1}\big)\nonumber%\label{equ_for_joint_distsd}\\
%.\end{align}
%We easily check that $\sum_{k=1}^m \sum_{k'=k}^m \distsd_m(\mbf{t},k,k')=1$ for any threshold $\mbf{t}$.

Finally, we will use the so-called Stirling numbers of the second kind, defined as coefficients $\Sti{s}{\l}$ for  $s,\l \geq 1$ by $\Sti{s}{0}=0$, $\Sti{s}{\l}=0$ for $\l>s$,  $\Sti{1}{1}=1$ and the recurrence relation: for all $1\leq \l \leq  s+1$,
%\begin{equation}
$
\Sti{s+1}{\l} = \l \Sti{s}{\l} + \Sti{s}{\l-1}.%\label{equ_coeff_Bls}
$
%\end{equation}
%Some values of $\Sti{s}{\l}$ are given in Table~\ref{table_Bls}. 
For instance, $\Sti{3}{1} = 1 $, $\Sti{3}{2} = 3 $, $\Sti{3}{3} = 1 $, $\Sti{4}{1} = 1 $,$\Sti{4}{2} = 7 $, $\Sti{4}{3} = 6 $, $\Sti{4}{4} = 1 $. 
From a combinatorial point of view, the coefficient $\Sti{s}{\l}$ counts the number of ways to partition a set of $s$ elements into $\l$ (nonempty) subsets. The latter is useful to compute the $s$-th moment of a binomial distribution: if $X\sim \mathcal{B}(n,q)$, we have $\forall s \geq 1$, $\E[ X^s]=\sum_{\l=1}^{s\wedge n} \frac{n!}{(n-\l)!} \Sti{s}{\l} q^\l$. % (e.g. this easily derives from Lemma~\ref{lem_Comb} in Appendix~\ref{notproof}). %, see also the related paper \cite{BM2005}).
%A recursive formula for moments of a binomial distribution (with S.M. Manago),
%College Mathematics Journal 36.1 (2005), 68-72.
%\begin{table}[h!]
%$$\begin{array}{c|ccccccc} s \backslash \l & 0& 1 & 2& 3 & 4& 5 & 6\\  \hline
%1 &0 & 1 &  0 &  0 &0&0&0 \\
%2 &0 & 1 & 1 &0&0&0&0\\
% 3 &0 & 1 & 3 & 1 &0&0&0\\
% 4 &0 & 1 & 7 & 6 & 1 &0&0\\
% 5 &0 & 1 & 15  & 25 & 10 &1 &0\\
%\end{array}$$
%\caption{\label{table_Bls} Values for $\Sti{s}{\l}$, $0\leq \l\leq s+1 \leq 6$.}
%\end{table}

\section{New formulas}
%In this section, we provide exact expressions for the FDR and the power, for any step-up and step-down procedures, and for the $s$-th moment of the FDP in the step-up case.
\subsection{Unconditional independent model, $m\geq 2$}\label{sec:mainresults}

\subsubsection{Step-up case}\label{sec:mainresults:stepup}

Let us consider the unconditional independent  model. Finner and Roters (2002) derived the exact distribution of the rejection number  of any step-up procedure in the case of i.i.d. uniform $p$-values (i.e., when all the hypotheses are true) \cite{FR2002}. %unconditional model, using that the $p$-values are in that case \textit{i.i.d}. 
In the unconditional model, the latter can be generalized as follows: denoting $G(t)=\pi_0 F_0(t)+(1-\pi_0)F_1(t)$ the common c.d.f. of the $p$-values, we have for $0\leq k \leq m$ that
\begin{equation}
 \prob{|\SU(\mbf{t})|=k} = \distsu_m\big( [G(t_{j})]_{1\leq j\leq m},k\big).\label{equ_dist_su}
 \end{equation}
(this is straightforward from  \cite{FR2002} %by using pointwise that $\SU(\mbf{t})(\mathbf{p})=\SU(G(\mbf{t}))(G(\mathbf{p}))$, 
because $G$ is continuous increasing). % (actually, their result stands for $\pi_0=1$, but the generalization is straightforward).
 
 Next, for the procedure $R(t)=\{i\telque p_i\leq t\}$ using a \textit{constant} threshold $t\in[0,1]$, Storey (2003) proved that the distribution of $|\cH_0(H)\cap R(t) |$ conditionally on $| R(t) |=k$ is a binomial distribution $\mathcal{B}\big(k, \pi_0 F_0(t)/G(t)\big)$ (see proof of Theorem~1 in  \cite{Storey2003}, see also Proposition 2.1 in \cite{Chi2008}). 
Later, Chi (2007) proved that  the distribution of $|\cH_0(H)\cap \LSU |$ conditionally on $| \LSU |=k$ is asymptotically binomial (in a particular ``supercritical'' framework), see \cite{Chi2007}.
 Here, we show that the latter holds non-asymptotically, for any step-up procedure, which, by using \eqref{equ_dist_su},  gives exact formulas for the c.d.f. of the FDP, the $s$-th moment of the FDP, the FDR and the Power.
 
\begin{theorem}\label{main_indep}
When testing $m\geq 2$ hypotheses, consider  a step-up procedure ${\SU}(\mbf{t})$ with threshold $\mbf{t}$ and the notation of Section~\ref{sec:notation}. Then for any parameter $\pi_0\in [0,1]$ and $F_1\in\mathcal{F}$, denoting $G(t)=\pi_0 F_0(t) + \pi_1 F_1(t)$, we have under the generating distribution ${(H,\mathbf{p})\sim \PExI_{(\pi_0,F_1)}}$ of the unconditional independent model, for any $k\geq 1$,
\begin{align}
|\cH_0(H)\cap \SU(\mbf{t}) |\: \mbox{ conditionally on }\: | \SU(\mbf{t}) |=k \:\:\:\:\sim\:\: \mathcal{B}\bigg(k, \frac{\pi_0 F_0(t_k)}{G(t_k)}\bigg).\label{equ_distrFDP}
\end{align}
In particular, we derive the following formulas, for any $x\in(0,1)$, for any $s\geq 1$, denoting by $\Sti{s}{\l}$ the Stirling number of second kind and by $ \lfloor z \rfloor$ the largest integer smaller than or equal to $z$,
\begin{align}
\P[\FDP(\SU(\mbf{t}),H) \leq x ]&= \sum_{k=0}^{m} \sum_{j=0}^{ \lfloor x k \rfloor} {{k} \choose {j}} \bigg(\frac{\pi_0F_0(t_k)}{G(t_k)}\bigg)^{j} \bigg(\frac{\pi_1F_1(t_k)}{G(t_k)}\bigg)^{k-j}  \distsu_m\big( [G(t_{j})]_{1\leq j\leq m},k\big)
%\sum_{k=0}^{m}{{m} \choose {k}}  \sum_{j=0}^{ \lfloor \alpha k \rfloor} {{k} \choose {j}} (\pi_0F_0(t_k))^{j}(\pi_1F_1(t_k))^{k-j}\Psi_{m-k}\big(1-G(t_m),...,1-G(t_{k+1})\big)
\label{equ_FDP_indep}\\
\E[\FDP(\SU(\mbf{t}),H)^{s} ]&= \sum_{\l=1}^{s \wedge m} \frac{m!}{(m-\l)!}\Sti{s}{\l} \pi_0^\l \sum_{k=\l}^{m} \frac{F_0(t_k)^\l}{k^s}  \: \distsu_{m-\l}\big(  [G(t_{j+\l})]_{1\leq j\leq m-\l} ,k-\l\big)%\nonumber ,\\
%&= \sum_{\l=1}^s \frac{m!}{(m-\l)!}\Sti{s}{\l} \pi_0^\l \sum_{k=\l}^{m} \frac{F_0(t_k)^\l}{k^s}  {{m-\l} \choose {k-\l}}  (G( t_k))^{k-\l} \Psi_{m-k}\big(1-G(t_m),...,1-G(t_{k+1})\big)
\label{equ_mom_FDP_indep} \\
\FDR(\SU(\mbf{t}),\PExI_{(\pi_0,F_1)})&=\pi_0 m \sum_{k=1}^{m} \frac{F_0(t_k)}{k}  \: \distsu_{m-1}\big(  [G(t_{j+1})]_{1\leq j\leq m-1} ,k-1\big)\label{equ_FDR_indep} \\
%&=\pi_0  \sum_{k=1}^{m} {{m} \choose {k}} F_0(t_k)  (G(t_k))^{k-1} \Psi_{m-k}\big( 1-G(t_m),...,1-G(t_{k+1})\big)
%\label{equ_FDR_indep_detailed}.
\Pow(\SU(\mbf{t}),\PExI_{(\pi_0,F_1)}) & =  \sum_{k=1}^{m} F_1( t_k)  \: \distsu_{m-1}\big([G(t_{j+1})]_{1\leq j\leq m-1},k-1\big). %\nonumber\\
%&=   \sum_{k=1}^{m}  {{m-1} \choose {k-1}} F_1( t_k) (G( t_k))^{k-1} \Psi_{m-k}\big(1-G(t_m),...,1-G(t_{k+1})\big).
\label{equ_Pow_indep}
\end{align}
\end{theorem}

We can apply Theorem~\ref{main_indep} in the case where $t_k=\alpha k /m$, to deduce the following results for the LSU procedure of Benjamini and Hochberg (1995) \cite{BH1995}: first,  \eqref{equ_FDR_indep} leads to $\FDR(\LSU)=\pi_0\alpha$,  recovering the well known result of Benjamini and Yekutieli \cite{BY2001} in the unconditional model. 
Second,  \eqref{equ_Pow_indep} provides the exact expression 
$$\Pow(\LSU)  =  \sum_{k=1}^{m} F_1( \alpha k/m)   {{m-1} \choose {k-1}}  (G(\alpha k/m))^{k-1} \Psi_{m-k}\big( 1-G(\alpha m/m),...,1-G(\alpha({k+1})/m)\big).$$
Glueck et al. (2008) have obtained an exact expression for the power of the LSU under independence (in the conditional model)  \cite{Glueck2008}, but the corresponding formula was reported to have a complexity exponential in $m$, which is intractable for large $m$. Here, we obtained a  much less complex formula, requiring (at most) the computation of the function $\Psi_{m}$ at one point of $\mathbb{R}^m$. 
Third,  formula \eqref{equ_mom_FDP_indep} used with $t_k=\alpha k/m$ is similar to expression (2.1) in Theorem~2.1 of \cite{FZ2006}, in which a formula for the $s$-th moment of the FDP of the $\LSU$ was investigated (in the conditional model). Our  formula \eqref{equ_mom_FDP_indep} uses additional factors $\Sti{s}{\l}$. As soon as $s\geq 3$, they are definitely needed and they seem forgotten in \cite{FZ2006}; for instance, taking $\alpha=1$, the corresponding linear step-up procedure rejects all the hypotheses and \eqref{equ_mom_FDP_indep} reduces to the computation of the $s$-th moment of a binomial distribution, which uses at least one $\Sti{s}{\l}> 1$, see Section~\ref{sec:notation}.
%As application, we make in Section~\ref{sec:LFCvar} a specific study of the variance of $\FDP(\LSU)$. 

Fourth, expression \eqref{equ_FDP_indep} used with $t_k=\alpha k/m$ yields what is to our knowledge the first exact expression for the c.d.f. of $\FDP(\LSU)$, valid for any $m\geq 2$ and for any alternative c.d.f. $F_1$. For instance, taking a typical Gaussian setting where $F_1(t)=\overline{\Phi} \left(  \overline{\Phi}^{-1}(t) - \mu \right)$, we are able to evaluate  the probability $\P(\FDP(\LSU)\leq c \:\alpha)$ for $c\geq 1$; for $\mu=3$, $\alpha=0.05$,  $\pi_0=1-1/\sqrt{m}$, expression \eqref{equ_FDP_indep} provides $\P(\FDP(\LSU)\leq \alpha) \simeq 0.724$, $\P(\FDP(\LSU)\leq 2\alpha) \simeq 0.787  $ for $m=100$ and  $\P(\FDP(\LSU)\leq \alpha) \simeq 0.557$, $\P(\FDP(\LSU)\leq 2\alpha) \simeq 0.826  $ for $m=1000$. 
This means that the LSU procedure, designed to control the FDR at level $\alpha$, can have a FDP larger than $2\alpha$ with a ``non-negligible" probability, for some admittedly quite standard values of the model parameters. 
As $m$ tends to infinity while the model parameters stay constant with $m$, the FDP converges to the FDR and the latter effect vanishes  (see e.g. \cite{Neu2008}). However, we show in Section~\ref{sec:LFCvar} that the convergence can be slow in the sparse case. %These remarks corroborate the Section~5.2 of \cite{Chi2008} where a sparse setting was provided in which the LSU procedure fails to control the FDP. 
    
  Alternatively, some authors are interested in procedures $R$ controlling the FDP, i.e. satisfying $\P[\FDP(R)\leq \alpha] \geq 1-\gamma$, see e.g. \cite{GW2004,LR2005}. As a matter of fact, by using \eqref{equ_distrFDP}, we directly deduce that the latter FDP control is satisfied by the step-up procedure $\SU(\mbf{t}^\star)$ using the \textit{oracle} threshold defined by $t^\star_{m+1}=1$ (by convention) and for any $1\leq k \leq m$, $$t^\star_k = \max\{ t\in[0,t^\star_{k+1}] \telque \P[X\leq \alpha k]\geq 1-\gamma\mbox{ for } X \sim \mathcal{B}(k, {\pi_0 t}/{G(t)}) \},$$ with $t^\star_k=0$ if the above set is empty. %(note that $t/G(t)$ should be assumed nondecreasing for the latter definition to be valid). 
  However, the latter threshold is unknown in practice, because it depends on the c.d.f. $G$ and an interesting issue is to estimate it. 
   Chi and Tan (2008) introduced  the threshold $t^{CT}_k = \max\{ t\in[0,t^{CT}_{k+1}] \telque \P[X\leq \alpha k]\geq 1-\gamma\mbox{ for } X \sim \mathcal{B}(k, 1\wedge (m t /k)) \}$ (with the convention $t^{CT}_{m+1}=1$). % for $k> k_m$ and $t^{CT}_k=0$ for $k\leq k_m$, with a tuning parameter $k_m$ to be well chosen
As a matter of fact, the latter can be seen as the empirical substitute of $t^\star_k$, because ${G}(t_{\hat{k}})\simeq \mathbb{G}_m(t_{\hat{k}})=\hat{k}/m$ for any step-up procedure rejecting $\hat{k}$ hypotheses ($\mathbb{G}_m$ denoting the e.c.d.f. of the $p$-values). 
Using the latter threshold, they established an asymptotic FDP control (as $m$ tends to infinity). Here, a plausible explanation is that their procedure correctly mimics the oracle  $\SU(\mbf{t}^\star)$  (asymptotically).
%The step-up procedure $\SU(\mathbf{t}^{TC})$ is proved to control the FDP asymptotically after a slight modification (see \cite{Chi2008}). 
%as suggested in \cite{Chi2008}. We are able to evaluate if it controls the FDP: often but not always !
%Montrer la convergence de  \cite{Chi2008} en quelques lignes en utilisant la nouvelle approche.
%seuils pas necessairement croissant cf seuilChicroissant dans les simus

%Several other consequences of Theorem~\ref{main_indep} are presented in Section~\ref{sec:deppossu} and in Section~\ref{appliLFC}. 

\subsubsection{Step-down case}\label{sec:indep-step-down}

In that section, we still consider the unconditional independent model, but we focus on the step-down case. By contrast with the step-up case, for a step-down procedure, the distribution of $|\cH_0(H)\cap \SD(\mbf{t}) |$  %(the false rejection number of the step-down procedure of threshold $\mbf{t}$) 
conditionally on $|\SD(\mbf{t})|=k$ % (the rejection number of the step-down procedure of threshold $\mbf{t}$) 
is not binomial, in general. For instance, we prove in Section~\ref{preuve_distr_SD}  that for $k\geq 1$, $k\leq m$,
\begin{align}
\P[|\cH_0(H)\cap \SD(\mbf{t}) |=k \telque |\SD(\mbf{t})=k]&=\pi_0^k \frac{\Psi_{k}\big(F_0(t_1),...,F_0(t_k))}{\Psi_{k}\big(G(t_1),...,G(t_k))} =: a^k\label{equ_distr_SD_1}\\
\P[|\cH_0(H)\cap \SD(\mbf{t}) |=0 \telque |\SD(\mbf{t})=k]&=\pi_1^k \frac{\Psi_{k}\big(F_1(t_1),...,F_1(t_k))}{\Psi_{k}\big(G(t_1),...,G(t_k))}=:(1-b)^k,\label{equ_distr_SD_2}
\end{align}
and it turns out that $a=b$ only for particular situations, such as $t_1=...=t_k$, $F_1(x)=x$ or $\pi_0\in\{0,1\}$. 
Also, in the Dirac-uniform configuration $F_1=1$ (an thus $G(t)=\pi_0 t +\pi_1$), we establish in Section~\ref{preuve_distr_SD}  that, for $1\leq j \leq k$,
\begin{align}
\P[|\cH_0(H)\cap \SD(\mbf{t}) |=j \telque |\SD(\mbf{t})=k]&={{k} \choose {j}}\pi_0^{j} \pi_1^{k-j} \frac{\Psi_{j}\big(t_{k-j+1},...,t_k)}{\Psi_{k}\big(\pi_0 t_1+\pi_1,...,\pi_0 t_k +\pi_1)}, \label{equ_distr_SD_3}
\end{align}
which is not binomial in general (see also Remark~\ref{rem-thorsten} below).

An exact expression for $\P[|\cH_0(H)\cap \SD(\mbf{t}) |=j \telque |\SD(\mbf{t})|=k]$ only involving functions of the type $\Psi_i$ and valid for any $j$, $\mbf{t}$ and $F_1$ seems considerably more difficult to derive. As a consequence, we  now only focus on the calculation of the FDR and of the power. For this, 
let us recall the exact formula for the distribution of $|\SD(\mbf{t})|$: for all $k\in \{0,...,m\}$, 
\begin{align}
 \prob{|\SD(\mbf{t})|=k} = \distsd_m\big( [G(t_{j})]_{1\leq j\leq m},k\big).\label{equ_dist_sd}
\end{align}
(see  formula (4) p.344 of \cite{SW1986} and \cite{FR2002} which can be directly generalized in the unconditional independent model because $G$ is continuous increasing). % in the case $\pi_0=1$, but the generalization is straightforward)
%To our knowledge, there is no exact formula for step-down procedures making a general link  between the FDR  and distributions of the type \eqref{equ_dist_sd}. 
The next result, based on the calculation of the distribution of $|\SD(\mbf{t}')|$ conditionally on $|\SD(\mbf{t})|=k$ (with $t'_j= t_{j+1}$),  connects the FDR to distributions of the type \eqref{equ_dist_sd} (see the proof in Section~\ref{proof_equ_jointdist_sd}).
 
\begin{theorem}\label{main_indep_sd}
For $m\geq 2$ hypotheses, consider the unconditional independent model $\PExI_{(\pi_0,F_1)}$,  a step-down procedure ${\SD}(\mbf{t})$  with threshold $\mbf{t}$ and the notation of Section~\ref{sec:notation}. Then for any parameter $\pi_0\in [0,1]$ and $F_1\in\mathcal{F}$ (denoting $G(t)=\pi_0 F_0( t) + \pi_1 F_1(t)$), we have
\begin{align}
\FDR(\SD(\mbf{t}),\PExI_{(\pi_0,F_1)})=&\pi_0  m \sum_{k=1}^{m}  \sum_{k'=k}^m \frac{F_0(t_k)}{k'}  
 \distsd_{m-1}\big((G(t_j))_{1\leq j\leq m-1},k-1\big)\: \nonumber\\
 &\times\distsd_{m-k} \left( \left(\frac{G(t_{k+j})-G(t_{k})}{1-G(t_{k})}\right)_{1\leq j\leq m-k},\:k'-k \right)
,\label{equ_FDR_indep_sd}\\
\Pow(\SD(\mbf{t}),\PExI_{(\pi_0,F_1)}) =&   \sum_{k=1}^{m} F_1( t_k)  \distsd_{m-1}\big( [G(t_{j})]_{1\leq j\leq m-1},k-1\big).\label{equ_Pow_indep_sd}
\end{align}
\end{theorem}

To the best of our knowledge, \eqref{equ_FDR_indep_sd} is the first exact expression for a step-down procedure that relies the FDR to the (joint) distribution of rejection numbers, for any $p$-value alternative distribution. The main tool to get this result is Lemma~\ref{lemma_down}.

One straightforward consequence of \eqref{equ_FDR_indep_sd} is that 
$ \pi_0  m \sum_{k=1}^{m}  \frac{t_k}{k}  \distsd_{m-1}\big( [G(t_{j})]_{1\leq j\leq m-1}
,k-1\big)$ is an upper-bound of $\FDR(\SD(\mbf{t}))$, which in particular proves that $\FDR(\SD(\mbf{t}))$ is always smaller than $\FDR(\SU(\mbf{t}))$, as soon as $t_k/k$ is nondecreasing in $k$. While the latter result should probably be considered as well known, it is not trivial because when increasing the rejection number, both numerator and denominator are increasing within the FDP expression (for further developments on this issue, see Theorem~2 in \cite{Zei2009}). 

To illustrate \eqref{equ_FDR_indep_sd}, we may consider the case where $t_k=\alpha k/m$, that is, we may compute exactly the FDR of the LSD procedure. %One well known result is that this procedure controls the FDR at level $\pi_0\alpha$ (see e.g. \cite{BR2008EJS}) in the conditional model and thus also unconditionally. This directly follows from \eqref{equ_FDR_indep_sd}, using that $F_0(t_k)\leq \alpha k'/m$. Additionally, \eqref{equ_FDR_indep_sd} gives the exact expression of FDR(LSD). 
In the Dirac-uniform model $F_1=1$ and using  \eqref{relFinner}, we deduce the following expression:
\begin{align}
\FDR(\LSD,\PExI_{(\pi_0,F_1=1)}) =& \pi_0 \frac{\alpha^2}{m} \left(\pi_1+\pi_0\frac{\alpha}{m}\right)  \sum_{k=1}^{m} \sum_{j=k}^m  \frac{k}{j}  {{m-1}\choose{k-1}}  {{m-k}\choose{j-k}} \pi_0^{m-k}\nonumber \\
&\times \left(\pi_1+\pi_0\frac{\alpha k}{m}\right)^{k-2}  \left(\frac{\alpha (j-k+1)}{m}\right)^{j-k-1} \left(1-\frac{\alpha (j+1)}{m}\right)^{m-j}.\label{equ_upperbound}
\end{align}
Finally, let us emphasize that 
%Some other applications of Theorem~\ref{main_indep_sd} are presented in Section~\ref{sec:deppossu} and in Section~\ref{appliLFC}. 
%In particular, 
expression \eqref{equ_FDR_indep_sd} is also useful to investigate least favorable configurations for the FDR of step-down procedures (see Section~\ref{appli_indep}).

\begin{remark}\label{rem-thorsten}
If we only focus on the Dirac-uniform configuration, the FDP distribution of a given procedure $R$ (rejecting any $p$-value equals to $0$) only depends on the distribution of $|R\cap \cH_0(H)|$ (conditionally on $H$), because $|R\cap \cH_0(H)| = |R| - m_1(H)$.  As done in Section~3.7 of \cite{Dick2008}, this leads to exact computations of the FDR for step-up, step-down and more general step-up-down procedures,   in the particular Dirac-uniform configuration.
In comparison, our approach is valid for an arbitrary alternative c.d.f. $F_1$, while it intrinsically uses the exchangeability of the $p$-values
(which requires to use an unconditional model). %in the unconditional model, intrinsically use the exchangeability of the $p$-values
%leads to different FDR expressions are different from those of  \cite{Dick2008} (even in the Dirac-uniform configuration) because they  . 
%dans le model cond:
%$\FDR(R) =m_0 \E\bigg[\frac{ t_{\tilde{k}_{(1)}+1}}{\tilde{k}'_{(1)}+1}\bigg].$
%Finally note that the Dirac uniform model is not a particularly interesting for the actual FDP control, because it is not a LFC.
\end{remark}

\subsection{Unconditional EMN model with nonnegative correlation, $m\geq 2$}\label{sec:deppossu}

In this subsection, our goal is to obtain results similar to Theorem~\ref{main_indep} and Theorem~\ref{main_indep_sd}, but this time in the unconditional EMN model of parameters $\pi_0$, $\rho$ and $\mu$, with a nonnegative correlation $\rho\in[0,1]$. 
In that case, we easily see that the joint distribution of the $p$-values can be realized as follows: for all $i$,
$p_i= \overline{\Phi} (\sqrt{\rho} \:\:{\overline{\Phi}}^{-1}(U) + \sqrt{1-\rho} \:\:{\overline{\Phi}}^{-1}(U_i) +\mu H_i)$, 
where $U$, $(U_i)_i$ are all i.i.d. uniform on $(0,1)$ (and independent of $(H_i)_i$). This idea can be traced back to Stuart (1958)  \cite{Stuart1958} and Owen and Steck (1962) \cite{OS1962}. As a consequence, conditionally on $U=u$, the $p$-values follow the unconditional independent model of parameters $\pi_0$, $F_0(\cdot\:|\:u,\rho)$ and $F_1(\cdot\:|\:u,\rho)$ where  we let
\begin{align}\label{form_pour_deppos} 
{F}_0(t \:|\:u,\rho) = \overline{\Phi} \left( \frac{{\overline{\Phi}}^{-1}(t) - \sqrt{\rho} \:\:{\overline{\Phi}}^{-1}(u) } {\sqrt{1-\rho}} \right)\mbox{, }
{F}_1(t\:|\:u,\rho) = \overline{\Phi} \left( \frac{{\overline{\Phi}}^{-1}(t) - \sqrt{\rho} \:\:{\overline{\Phi}}^{-1}(u) -\mu} {\sqrt{1-\rho}} \right)%\label{form1_pour_deppos}
\end{align}
for $\rho\in[0,1)$ and ${F}_0(t \:|\:u,1) =\ind{u\leq t}$, ${F}_1(t \:|\:u,1) =\ind{u\leq {\overline{\Phi}}({\overline{\Phi}}^{-1}(t)-\mu)} $ for $\rho=1$. As a result, to obtain formulas valid in the unconditional EMN model, we may directly use formulas holding in the unconditional independent model (with the above modified c.d.f.'s) and using an additional integration over $u\in(0,1)$. 
%To illustrate this method, let us first consider the case of a constant thresholding $\mbf{t}(.)=t \in [0,1]$. In that case, we obtained in the previous subsection that 
%$\FDR(t, \PExI_{\pi_0,F_1})=  \pi_0 \frac{t}{G(t)} (1-(1-G(t))^{m})$. Hence, the latter reasoning yields, in the unconditional EMN model,
%$\FDR(t,\PExN_{(\pi_0,\rho,\mu)})=  \pi_0 t \int_0^1 \frac{1-(1-G(t\:|\:u,\rho))^{m}}{G(t\:|\:u,\rho)} \:du $, where $G(t\:|\:u,\rho)=\pi_0{F}_0(t \:|\:u,\rho)+\pi_1 {F}_1(t \:|\:u,\rho)$ is defined using \eqref{form_pour_deppos}. 
%More generally, 
Hence, we deduce from Theorem~\ref{main_indep} and Theorem~\ref{main_indep_sd} the following result (the formulas are not fully written for short):

\begin{corollary}\label{main_deppos}\label{main_deppos_sd}
For $m\geq 2$ hypotheses, consider the unconditional EMN model $\PExN_{(\pi_0,\rho,\mu)}$ with parameters  $\pi_0\in[0,1]$, $\mu>0$ and $\rho\in[0,1]$ and let $G(t\:|\:u,\rho)=\pi_0{F}_0(t \:|\:u,\rho)+\pi_1 {F}_1(t \:|\:u,\rho)$  using notation \eqref{form_pour_deppos}. Then, for any threshold $\mbf{t}$, under the generating distribution $(H,\mathbf{p})\sim \PExN_{(\pi_0,\rho,\mu)}$, the quantity $\P[\FDP(\SU(\mbf{t})) \leq x ]$ (resp. $\E[\FDP(\SU(\mbf{t}))^s ]$;  $\FDR(\SU(\mbf{t}))$; $\Pow(\SU(\mbf{t}))$) is given by the RHS of \eqref{equ_FDP_indep} (resp.\eqref{equ_mom_FDP_indep}; \eqref{equ_FDR_indep};  \eqref{equ_Pow_indep}), by replacing $F_0(\cdot)$ by ${F}_0(\cdot \:|\:u,\rho)$, $F_1(\cdot)$ by ${F}_1(\cdot \:|\:u,\rho)$ and $G(\cdot)$ by ${G}(\cdot \:|\:u,\rho)$, and by integrating over $u$ with respect to the Lebesgue measure on $(0,1)$.
Additionally, a similar result holds for step-down procedures using \eqref{equ_FDR_indep_sd} and \eqref{equ_Pow_indep_sd}.
\end{corollary}

In particular, applying Corollary~\ref{main_deppos_sd} for the LSU procedure, we obtain that $\FDR(\LSU,\PExN_{(\pi_0,\rho,\mu)})$ equals:
\begin{equation}
\pi_0  \sum_{k=1}^{m} {{m} \choose {k}} \int_0^1 F_0(\alpha k/m| u,\rho)  G(\alpha k/m| u,\rho)^{k-1} \Psi_{m-k}\big( (1-G(\alpha (m-j+1)/m | u,\rho))_{1\leq j \leq m-k}\big) du.
\label{equ_FDR_posdep_detailed} 
\end{equation}
An expression for $\lim_m \FDR(\LSU,\PExN_{(\pi_0,\rho,\mu)})$ was provided by Finner et al. (2007),  by considering the asymptotic framework where $m$ tends to infinity \cite{FDR2007}. We  compared the latter to the formula \eqref{equ_FDR_posdep_detailed} by plotting the graph corresponding to their Figure~3 (not reported here). The results are qualitatively the same for $\pi_0<1$, but present major differences when $\pi_0=1$ and $\rho$ is small. This is in accordance with the simulations reported in the concluding remarks of Section~5 in \cite{FDR2007}. Hence, the asymptotic analysis may not reflect what happens for a realistically finite $m$, which can be seen as a limitation with respect to our non-asymptotic approach.
As illustration,  when $\pi_0=1$, Finner et al. (2007) proved that   $\lim_{\rho\rightarrow 0}\lim_m  \FDR(\LSU) = \overline{\Phi}(\sqrt{-2\log{\alpha}})<\alpha$ whereas we have $  \lim_{\rho\rightarrow 0} \FDR(\LSU) = \alpha$, as remarked in \cite{FDR2007} using simulations and  as proved formally in the next result.
%Furthermore, when $\pi_0=1$, Finner et al. (2007) proved that   $\lim_{\rho\rightarrow 0}\lim_m  \FDR(\LSU) = \overline{\Phi}(\sqrt{-2\log{\alpha}})<\alpha$ and they put forward the hypothesis that  $\lim_m  \lim_{\rho\rightarrow 0} \FDR(\LSU) = \alpha$ (using simulations), to conclude that the order of limits plays a crucial role. The next result proves that the latter supposition is true because for each $m$, $\FDR(\LSU)$ is continuous in $\rho$.
\begin{corollary}\label{cor_regul}
For any $m\geq 2$ and for any threshold $\mbf{t}$, %and for any parameter $\pi_0\in[0,1]$ and $\mu>0$, 
the quantities $\FDR(SU(\mbf{t}), \ol{P}^N_{(\pi_0,\rho,\mu)})$ and $\FDR(\SD(\mbf{t}), \ol{P}^N_{(\pi_0,\rho,\mu)})$ are continuous in any $\pi_0\in[0,1]$, any $\rho\in [0,1]$ and any $\mu>0$. %Moreover, when $m=2$, these functions are continuous in $\rho=-1$ as well.
\end{corollary}

Corollary~\ref{cor_regul} is a straightforward consequence of Corollary~\ref{main_deppos_sd}; %is that we may prove that the FDR is continuous in any $\rho\in [0,1]$ (for any step-up or step-down procedure), which corroborates what was observed on  simulations (see e.g. cite{BR2008b}).
%For this, 
to prove the continuity in $\rho=1$, we may remark that 
for any $u$ outside the set $\mathcal{S}=\{t_k,1\leq  k\leq m\}\cup\{\ol{\Phi}(\ol{\Phi}^{-1}(t_k)-\mu),1\leq k\leq m\}$ of zero Lebesgue measure, the functions ${F}_0(t \:|\:u,\rho) $ and ${F}_1(t\:|\:u,\rho)$ are continuous in $\rho=1$. %The following result thus holds.

In particular, Corollary~\ref{cor_regul} shows that the limit of the FDR when $\rho$ tends to $1$ is given by the degenerated case $\rho=1$. In the latter case, the FDR is particularly easy to compute because only one Gaussian variable is effective: for step-up procedures,
%\begin{align}
$\FDR(\SU(\mbf{t})) = \pi_0 t_m$; % \nonumber\\
%$\Pow(\SU(\mbf{t})) =  \sum_{k=1}^{m}  {{m-1} \choose {k-1}} \pi_1^{k-1} \pi_0^{m-k} \max\{ t_m, \overline{\Phi} ( {\overline{\Phi}}^{-1}(t_k) -\mu ) \}$; %\nonumber\\
and for step-down procedures,
$\FDR(\SD(\mbf{t})) =\sum_{k=1}^{m}  {{m} \choose {k}} \pi_0^{k} \pi_1^{m-k} \frac{k}{m}\min\{ t_{m-k+1}, \overline{\Phi} ( {\overline{\Phi}}^{-1}(t_1) -\mu ) \}$ %\nonumber\\
%$\Pow(\SD(\mbf{t})) = \ol{\Phi}({ \ol{\Phi}}^{-1}(t_1)-\mu).$% \nonumber
%\end{align}
(the proof is left to the reader). For instance, under the special $p$-value configuration where $\rho=1$ and $\pi_0=1$, the above FDR expressions yield $\FDR(\SU(\mbf{t}))=t_m$ and $\FDR(\SD(\mbf{t}))=t_1$ . Thus, as $\rho\simeq 1$, the FDR value may considerably change as one considers a step-up or a step-down algorithm.

Going back to Corollary~\ref{main_deppos_sd}, let us mention that the latter can be used in order to evaluate the FDR control robustness under Gaussian equicorrelated positive dependence for any  procedure (step-up or step-down)  that controls the FDR under independence. For instance, the  \textit{adaptive} procedures of Blanchard and Roquain (2008) \cite{BR2008b} (step-up using $t_k= \alpha \min\{1,(1-\alpha)k/(m-k+1)\}$) and Finner et al. (2009) \cite{FDR2009} (step-up based upon $t_k= \alpha k/(m-(1-\alpha)k)$) %Gavrilov et al. (2009) \cite{GBS2009} (step-down using $t_k= \alpha k/(m-(1-\alpha)k+1)$) 
have been proved to control the FDR at level $\alpha$ under independence (asymptotically for \cite{FDR2009}). 
 A simulation study was done in \cite{BR2008b} in order to check if their respective FDR is still below $\alpha$ (or at least close to $\alpha$) in the EMN model. 
Using our exact approach, we are able to reproduce their analysis without %, but using exact formulas hence avoiding 
the errors due to the Monte-Carlo approximation. However, we underline that our approach uses \textit{non-random} thresholds $\mbf{t}$; this is not always the case for adaptive procedures (see e.g. \cite{BKY2006,BR2008b}). %We did not report the corresponding graphs to save some space.

\subsection{EMN model with  a general correlation and $m=2$}\label{sec:depneg}

When the correlation $\rho$ is negative, the approach presented in the last section is not valid anymore  and the problem seems considerably more difficult to tackle. We propose in this section to focus on the case where only two hypotheses are tested,  which should hopefully give some hints concerning the behavior of the FDR under negative correlations for larger $m$. The next result follows from elementary integration and does not require the use of an unconditional model.

\begin{proposition}\label{th_m2}
For $m = 2$ hypotheses, consider the conditional EMN model  $P^N_{(H,\rho,\mu)}$ with parameters $H=(H_1,H_2)\in\{0,1\}^2$ (generating $m_0\in \{1,2\} $ true null hypotheses), $\rho\in [-1,1]$ and $\mu>0$. Consider a threshold $\mbf{t}=(t_1,t_2)$. Let $z_1=\ol{\Phi}^{-1} (t_2) $ and $ z_2=\ol{\Phi}^{-1} (t_1)$. Then $\FDR(\SU(\mbf{t}),P^N_{(H,\rho,\mu)})$ is given by 
{\small
\begin{equation}
\begin{array}{|l|l|}
\hline
\rho\in(-1,1), m_0=1 & \frac{1}{2}  \int_{0}^{\ol \Phi (z_1-\mu )} \ol \Phi\left(\frac{z_1-\rho \ol{\Phi}^{-1} (w)}{\sqrt{1-\rho^2}}\right) \,dw  +  \int_{\ol \Phi (z_1-\mu )}^{1} \ol  \Phi\left(\frac{z_2-\rho \ol{\Phi}^{-1} (w)}{\sqrt{1-\rho^2}}\right) \,dw\\
\rho\in(-1,1), m_0=2 & t_1 +  \int_{t_1}^{t_2} \ol \Phi\left(\frac{z_1 -\rho \ol{\Phi}^{-1} (w)}{\sqrt{1-\rho^2}}\right) \,dw  +  \int_{t_2}^{1} \ol \Phi\left(\frac{z_2-\rho \ol{\Phi}^{-1} (w)}{\sqrt{1-\rho^2}}\right) \,dw
\\
\hline
\rho=-1, m_0=1 &\left\lbrace
\begin{array}{lll}
t_1 & \mbox{if} & 0 < \mu \le 2 z_1\\
t_1+\frac{1}{2}t_2-\frac{1}{2}\ol{\Phi}(\mu-z_1) & \mbox{if} & 2
z_1< \mu < z_1 +z_2\\
\frac{1}{2}t_2+\frac{1}{2}\ol{\Phi}(\mu-z_1)& \mbox{if} &  \mu \ge  z_1 +z_2
\end{array}
\right.
\\
\rho=-1, m_0=2 & \left\lbrace
\begin{array}{lll}
2 t_1 & \mbox{if} & 1/2 \geq t_2\\
2(t_1+t_2)-1& \mbox{if} & 1/2<t_2, t_1+t_2\leq 1\\
1& \mbox{if} & 1/2<t_2, t_1+t_2> 1
\end{array}
\right.
\\
\hline
\rho=1, m_0=1 & \frac{1}{2} t_2 \\
\rho=1, m_0=2 & t_2 \\
\hline
\end{array}
\label{equ_FDR_SU_m2},
\end{equation}
}
and $\FDR(\SD(\mbf{t}),P^N_{(H,\rho,\mu)})$ is given by
{\small
\begin{equation}
\begin{array}{|l|l|}
\hline
\rho\in(-1,1), m_0=1 & \frac{1}{2} \int_{0}^{\ol \Phi (z_2-\mu  )} \ol \Phi\left(\frac{z_1-\rho \ol{\Phi}^{-1} (w)}{\sqrt{1-\rho^2}}\right) \,dw +\frac{1}{2} \int_{\ol \Phi (z_2-\mu )}^{\ol \Phi (z_1-\mu )} \ol \Phi\left(\frac{z_2-\rho \ol{\Phi}^{-1} (w)}{\sqrt{1-\rho^2}}\right) \,dw \\
&+\int_{\ol \Phi (z_1-\mu )}^{1} \ol \Phi\left(\frac{z_2-\rho \ol{\Phi}^{-1} (w)}{\sqrt{1-\rho^2}}\right) \,dw\\
\rho\in(-1,1), m_0=2 & t_1+\int_{t_1}^{1} \ol \Phi\left(\frac{z_2-\rho \ol{\Phi}^{-1} (w)}{\sqrt{1-\rho^2}}\right) \,dw\\
\hline
\rho=-1, m_0=1 &\left\lbrace
\begin{array}{lll}
t_1 & \mbox{if} & 0 < \mu \le z_1 +z_2\\
\frac{1}{2}(t_1+ t_2) -\frac{1}{2}\ol{\Phi}(\mu-z_2)+\frac{1}{2}\ol{\Phi}(\mu-z_1)  & \mbox{if} & z_1+z_2< \mu < 2z_2\\
\frac{1}{2}t_2+\frac{1}{2}\ol{\Phi}(\mu-z_1)& \mbox{if} &  \mu \ge  2z_2
\end{array}
\right.
\\
\rho=-1, m_0=2 & \min(2t_1,1)
\\
\hline
\rho=1, m_0=1 & \frac{1}{2} \min( t_2, \ol{\Phi}(z_2-\mu) )\\
\rho=1, m_0=2 & t_1 \\
\hline
\end{array}
\label{equ_FDR_SD_m2}.
\end{equation}
}
\end{proposition}\vspace{0.5cm}

\section{Application to least/most favorable configurations}\label{appliLFC}

\subsection{Least favorable configurations for the FDR}\label{sec:LFCFDR}

In order to study the FDR control, an interesting multiple testing issue is to determine which are the values of the model parameter $F_1$ (or $\mu$) for which the FDR is maximum. The latter is called a least favorable configuration (LFC) for the FDR. 

\subsubsection{Independent model} \label{appli_indep}

Let us focus on the unconditional independent model.
For a \textit{step-up} procedure, expression \eqref{equ_FDR_indep} can be seen as 
$\pi_0 m \E_{0}[ t_{\hat{k}}/\hat{k} ]$ (where $\hat{k}=|\SU([G(t_{j+1})]_{1\leq j\leq m-1})|+1$ and $\E_0$ denotes the expectation with respect to $m-1$ i.i.d. uniform $p$-values). This shows that the behavior of the function $k\mapsto t_k/k$ is crucial to determine the LFC's of the FDR. Namely, if $F_1$ and $F_1'$ are two c.d.f.'s such that for all $t\in[0,1]$, $F_1(t)\leq F_1'(t)$, then we have $\FDR(\SU(\mbf{t}),\PExI_{(\pi_0,F_1)}) \leq \FDR(\SU(\mbf{t}),\PExI_{(\pi_0,F'_1)})$ when  $t_k/k$ is nondecreasing, while we have $\FDR(\SU(\mbf{t}),\PExI_{(\pi_0,F_1)})  \geq $ $\FDR(\SU(\mbf{t}),\PExI_{(\pi_0,F'_1)})$ when  $t_k/k$ is nonincreasing. The latter recovers a well known result of Benjamini and Yekutieli (2001) which was initially established in the conditional model, see Theorem~5.3 in
%, , we then have $\FDR(\SU(\mbf{t}),P^I_{(H,F_1)}) \leq \FDR(\SU(\mbf{t}),P^I_{(H,F_1')})$ (and thus $\FDR(\SU(\mbf{t}),\PExI_{(\pi_0,F_1)}) \leq \FDR(\SU(\mbf{t}),\PExI_{(\pi_0,F'_1)})$
\cite{BY2001}.
As a consequence, the LFC for the FDR is either $F_1=1$ (Dirac-uniform)  when  $t_k/k$ is nondecreasing, or $F_1=0$ ($F_1(x)=x$ if we only look at concave c.d.f.'s) when  $t_k/k$ is nonincreasing. 
In the border case of a linear threshold, the FDR  does not depend on $F_1$ (e.g. it is equal to $\pi_0 \alpha$ for the LSU), hence any configuration is a LFC.

%for any step-up procedure $\SU(\mbf{t})$ satisfying that $t_k/k$ is nondecreasing, the LFC is the Dirac-uniform configuration $F_1=1$. As a matter of fact, they proved more than that: if $F_1$ and $F_1'$ are two c.d.f.'s with for all $t$, $F_1(t)\leq F_1'(t)$, we then have $\FDR(\SU(\mbf{t}),P^I_{(H,F_1)}) \leq \FDR(\SU(\mbf{t}),P^I_{(H,F_1')})$ (and thus $\FDR(\SU(\mbf{t}),\PExI_{(\pi_0,F_1)}) \leq \FDR(\SU(\mbf{t}),\PExI_{(\pi_0,F'_1)})$ unconditionally; the latter can also be directly derived from \eqref{equ_FDR_indep}). For the LSU, the FDR  does not depend on $F_1$ under independence (it is equal to $\pi_0 \alpha$) hence any configuration is a LFC.

An open problem is to determine the LFC's of a  \textit{step-down} procedure using a given threshold $\mbf{t}=(t_j)_{1\leq j \leq m}$. 
Here, we introduce a new condition on the threshold $\mbf{t}$ which provides that the Dirac-uniform configuration is  a LFC for the FDR of the corresponding \textit{step-down} procedure.
We define for any threshold $\mbf{t}=(t_j)_{1\leq j \leq m}$ the following condition:
\begin{align}
k\in\{1,...,m\} \mapsto \sum_{i=0}^{m-k} \frac{t_k}{k+i}
\distsd_{m-k} \left( \left(\frac{t_{k+j}-t_k}{1-t_k}\right)_{1\leq
    j\leq m-k},i \right)\mbox { is nondecreasing}
     \label{fun_croiss} \tag{$A$}.
\end{align}
We now present the main result of this section, which uses Theorem~\ref{main_indep_sd} and is proved in Section~\ref{proof_appli_indep_sd}.

\begin{theorem}\label{appli_indep_sd}
For $m\geq 2$ hypotheses, consider the unconditional independent model $\PExI_{(\pi_0,F_1)}$ and a step-down procedure ${\SD(\mbf{t})}$ with a threshold $\mbf{t}=(t_j)_{1\leq j \leq m}$ satisfying \eqref{fun_croiss}. 
Then for any $\pi_0\in [0,1]$ and concave c.d.f. $F_1\in\mathcal{F}$, we have
\begin{align}
\FDR(\SD(\mbf{t}),\PExI_{(\pi_0,F_1)})&\leq \FDR(\SD(\mbf{t}),\PExI_{(\pi_0,F_1=1)})\label{equ_LFC_LSD},
\end{align}
meaning that the Dirac-uniform distribution is a least favorable configuration for the FDR of $\SD(\mbf{t})$. Moreover, for $\alpha\in (0,1)$, the linear threshold $\mbf{t}=(\alpha j/m)_{1\leq j\leq m}$ satisfies  \eqref{fun_croiss} and thus \eqref{equ_LFC_LSD} holds for the linear step-down procedure LSD.
\end{theorem}

%The above result states that  \eqref{fun_croiss} is sufficient to get \eqref{equ_LFC_LSD}. 

%Let us now discuss condition \eqref{fun_croiss}. First,
While condition \eqref{fun_croiss}  may be somehow difficult to state formally, %(e.g., in the case of the linear threshold, see the proof in Section~\ref{proof_appli_indep_sd}), 
it is very easy to check numerically because it only involves a finite set of real numbers. For instance, considering the threshold of Gavrilov et al. (2009)   $t_k=\alpha k /(m+1 -(1-\alpha)k)$ (for which the step-down procedure controls the FDR, see \cite{GBS2009}), we may see that $(t_k)_k$ satisfies \eqref{fun_croiss} for each $(\alpha,m) \in \{0.01,0.05,0.1,0.2,0.5,0.9\}\times\{5,10,50,100\}$, for instance.  In fact, we were not able to find a value of $(\alpha,m)$ for which the corresponding threshold does not satisfy \eqref{fun_croiss} (unfortunately, we have yet no formal argument proving \eqref{fun_croiss}   for any value of $(\alpha,m)$). %which applies to \textit{any} $(\alpha,m)\in[0,1]\times \{1,...,m\}$). 
As a consequence, Theorem~\ref{appli_indep_sd} states that the LFC for the procedure of Gavrilov et al. (2009)  is still the Dirac-uniform configuration (over the class of concave c.d.f.'s), at least for the previously listed values of $(\alpha,m)$, which is a new interesting finding.

In comparison with the step-up case, for which the standard condition ``$k\mapsto t_k/k$ nondecreasing'' provides that the Dirac-uniform configuration is a LFC, the new sufficient condition \eqref{fun_croiss} in the step-down case may be written as `` $k \mapsto t_k/k \:.\:\Psi(\mathbf{t},k)$ is nondecreasing'', where $\Psi(\mathbf{t},k)=\sum_{i=0}^{m-k} \frac{k}{k+i} \distsd_{m-k} \left( \left(\frac{t_{k+j}-t_k}{1-t_k}\right)_{1\leq    j\leq m-k},i \right)$. It turns out that the additional function $\Psi$ has a quite complex behavior, not necessarily connected to the one of $t_k/k$, so that there is no general relation between \eqref{fun_croiss} and ``$k\mapsto t_k/k$ nondecreasing''; for instance,  on the one hand, \eqref{fun_croiss} does not hold for 
the piecewise linear threshold defined by $t_k=\alpha p k /m$ for $1\leq k \leq a$ and  by $t_k= \alpha (p a-m)  ((a-m) m)^{-1} k  + \alpha(1- (p a -m)(a-m)^{-1})$ for $a+1\leq k\leq m$
%$t_k=\alpha k p/m \ind{1\leq k \leq a } + \alpha [ k (p a-m)/m +a(1-p) ]/(a-m) \ind{a+1\leq k\leq m}$ 
(using e.g. $m=50$, $p=0.6$, $a=4$, $\alpha=0.5$), while $t_k/k$ is nondecreasing. On the other hand, \eqref{fun_croiss} holds for $t_k=0.9\:(k/m)^{9/10}$ (using e.g. $m=50$) while $t_k/k$ is decreasing. 
%As a result, \eqref{equ_LFC_LSD} may hold although $k\mapsto t_k/k$ is decreasing. %Moreover, \eqref{equ_LFC_LSD} may not be true although $k\mapsto t_k/k$ is nondecreasing  (e.g. take ??).
%However, it turns out that a threshold $t_k$ such that $t_k/k$ is decreasing may also enjoy \eqref{fun_croiss}, for instance by taking $t_k=0.9\:(k/m)^{9/10}$ and $m=50$. Thus, Theorem~\ref{appli_indep_sd} applies and the LFC for the corresponding step-down procedure is the Dirac-uniform configuration $F_1=1$.  while the LFC for the corresponding step-up procedure  is the configuration $F_1=0$
In particular, for the latter threshold and considering only the set of concave c.d.f.'s, a LFC for $\FDR(\SU(\mbf{t}))$ is $F_1(x)=x$ while a LFC for $\FDR(\SD(\mbf{t}))$ is $F_1=1$ (and we checked numerically using Gaussian models that $F_1(x)=x$ is not a LFC for  $\FDR(\SD(\mbf{t}))$\,). This puts forward the complexity of the issue: whether we consider a step-up or a step-down procedure, the LFC's for the FDR may be different for some thresholds (e.g. $t_k=0.9\:(k/m)^{9/10}$, $m=50$) and they may coincide for some other thresholds (e.g. the one of \cite{GBS2009} for suitable $(\alpha,m)$). 

\subsubsection{EMN model} \label{appli_dep}

When the $p$-values follow the EMN model, Finner et al. (2007) conjectured that a LFC for the FDR of the LSU is still the Dirac-uniform distribution  (see Section~1 in \cite{FDR2007}). Here, we support this conjecture when $\rho\geq 0$ but we disprove it when $\rho<0$.

In order to investigate this issue, we reported on Figure~\ref{DUpasPireCas} the FDR of the LSU procedure against $\mu$ in the EMN model when $\rho> 0$ (left) and when $\rho<0$ (right), by using Corollary~\ref{main_deppos} and Proposition~\ref{th_m2}.
\begin{figure}[h!]
\begin{center}
\includegraphics[width=0.4\textwidth,angle=-90]{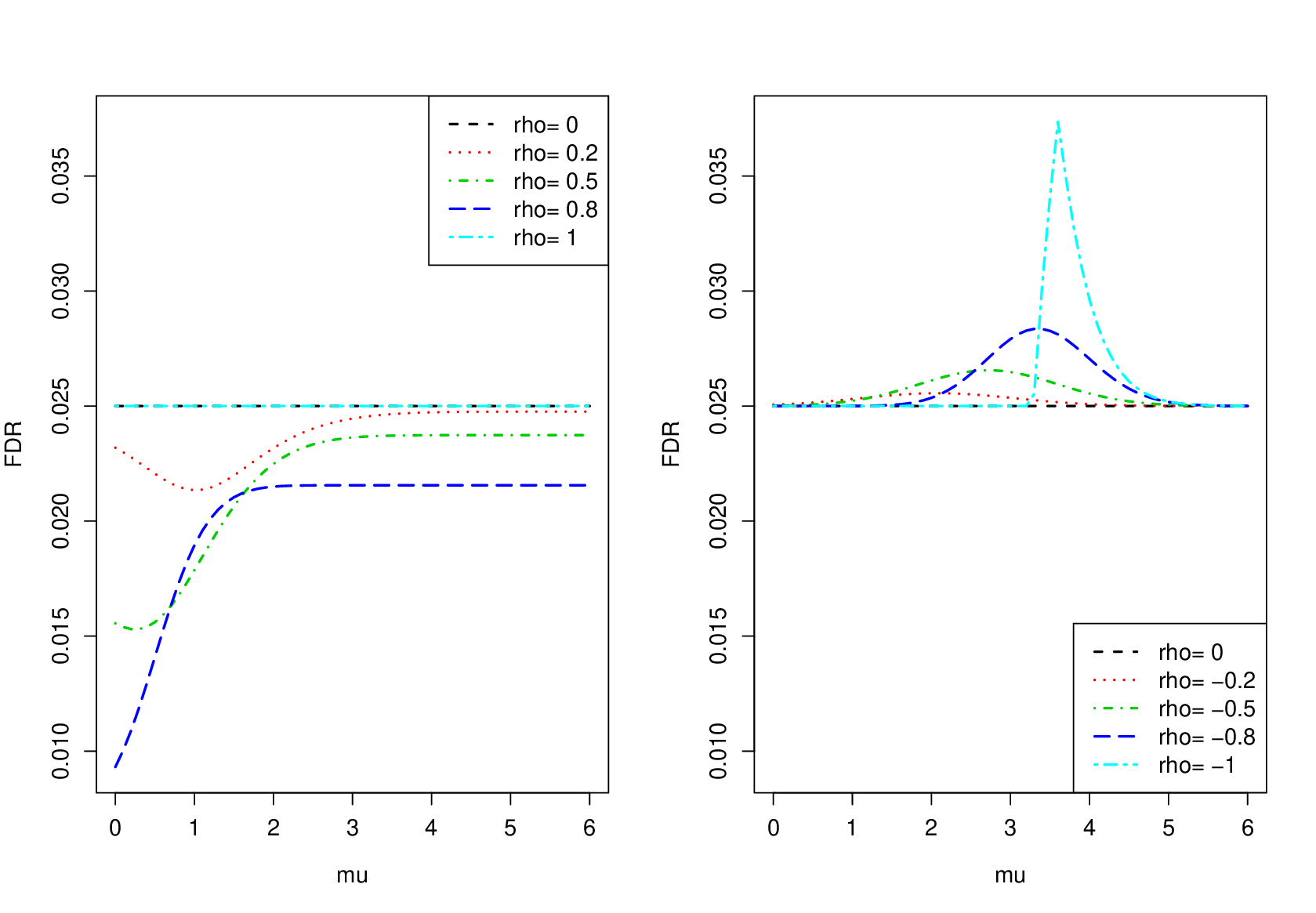}
\end{center}%\vspace{-5mm}
\caption{FDR(LSU) against the mean $\mu$. Left: $\rho\geq 0$ unconditional EMN model $m=100$ and $\pi_0=0.5$. Right: $\rho<0$ conditional EMN model with $m=2$ and $m_0=1$. $\alpha=0.05$.}  \label{DUpasPireCas}
\end{figure}
Under positive correlation, although each curve is not necessarily nondecreasing (e.g. for $\rho=0.2$), the case $\mu=\infty$, close to the right most point of Figure~\ref{DUpasPireCas}, seems to be a LFC. A challenging problem  is to state the latter formally. % (we were unfortunately not able to prove the latter formally). 
Under negative correlation and $m=2$, however,  $\mu=\infty$ is not a LFC anymore. As a matter of fact,  in the case where $m=2$, $m_0=1$ and $\mu=\infty$, the two $p$-values are independent (one $p$-value equals $0$), so that the FDR equals $\alpha/2=\alpha m_0/m$ which is not a maximum for the FDR, as we will show below.  

Qualitatively, we observed the same behavior concerning the FDR of the LSD procedure.\\

Under negative correlation,  the Dirac-uniform is not a LFC for the FDR and we can therefore legitimately ask what are the LFC's in that case. Here, we propose to solve this problem when $m=2$ in the conditional EMN model. Let $z_1=\ol{\Phi}^{-1}(\alpha)$ and $z_2=\ol{\Phi}^{-1}(\alpha/2)$ and first consider the LSU procedure. Its FDR is plotted in Figure~\ref{negcorr} (top). When $m_0=1$, we can check that $(\rho,\mu)=(-1,z_1+z_2)$ is a LFC because applying \eqref{equ_FDR_SU_m2}, the corresponding FDR is $3\alpha/4$, which equals the Benjamini-Yekutieli's (BY) upper-bound $(1+1/2+...+1/m)\alpha m_0/m$ \cite{BY2001} (valid under any dependency). Interestingly, in general, Guo and Rao (2008)  states that the BY bound can be fulfilled using very specific dependency structures between the $p$-values (not necessarily including those coming from a EMN model) \cite{GR2008}. Here, we remark that the maximum value of the FDR   still equals the BY bound for $m_0=1$, even for $p$-values coming from a EMN model. %This may be an artifact due to the fact that we consider only $m=2$ null hypotheses.
Next, for $m_0=2$, we may differentiate the corresponding expression in \eqref{equ_FDR_SU_m2} in order to obtain that, assuming $\alpha\leq 1/2$,  the FDR (that does not depend on $\mu$) attains its maximum in $\rho=(-z_1(z_1- z_2) - \{(z^2_1 - z_1 z_2)^2 + 2  \log(2) (z^2_1 - z^2_2) +4 \log(2)\}^{1/2})/(2\log(2))\in(-1,0)$.

Second, we consider the LSD procedure, whose FDR is plotted in Figure~\ref{negcorr} (bottom). In the case where $m_0=2$, by differentiating 
the corresponding expression in \eqref{equ_FDR_SU_m2} (that does not depend on $\mu$), 
%$ - (2\pi\sqrt{1-\rho2})^{-1} \exp( -\frac{z2_2}{1+\rho}) $
we are able to state that the FDR attains its maximum at $\rho=-1$ and that the value of the maximum is $\alpha$. In particular, the FDR of the LSD procedure is always smaller than or equal to $\alpha$ when $m=2$, in the conditional EMN model (for any $m_0$) and thus also in the unconditional EMN model, even for a negative correlation. An interesting open problem is to know whether this holds for larger $m$. %is specific to the case $m=2$.

\begin{figure}[h!]
\begin{center}
\includegraphics[width=0.24\textwidth,angle=-90]{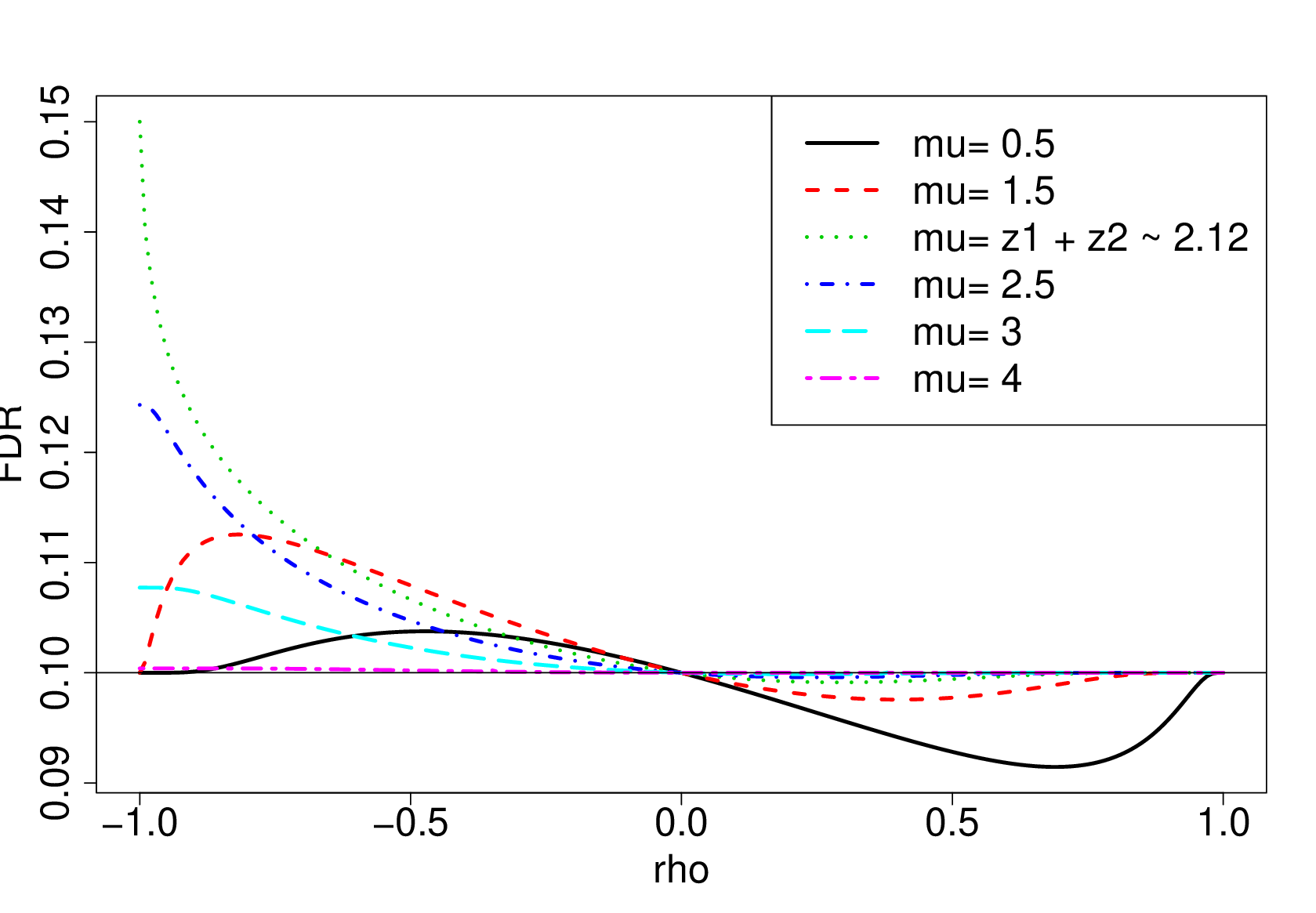}
\includegraphics[width=0.24\textwidth,angle=-90]{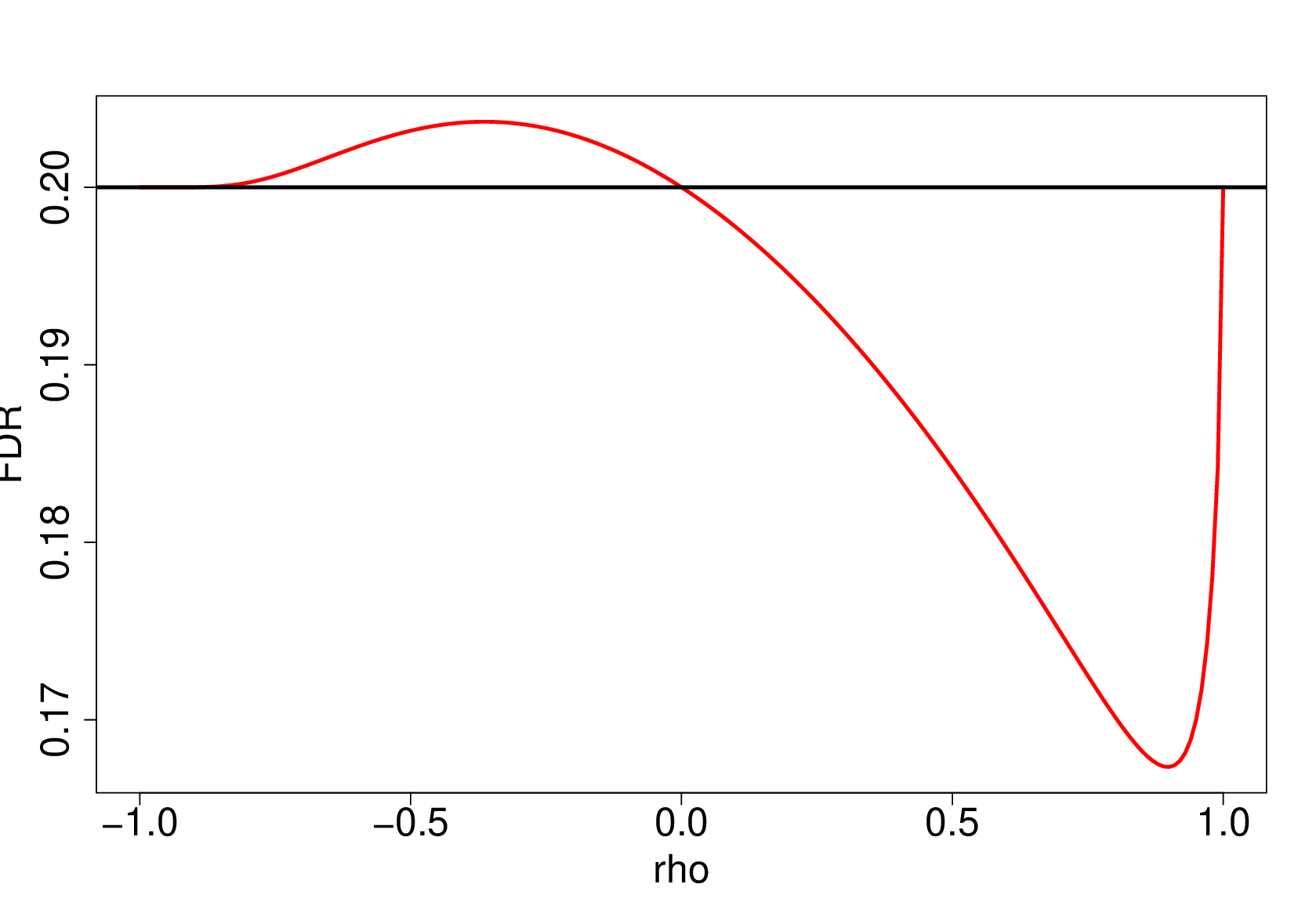}\\
\includegraphics[width=0.24\textwidth,angle=-90]{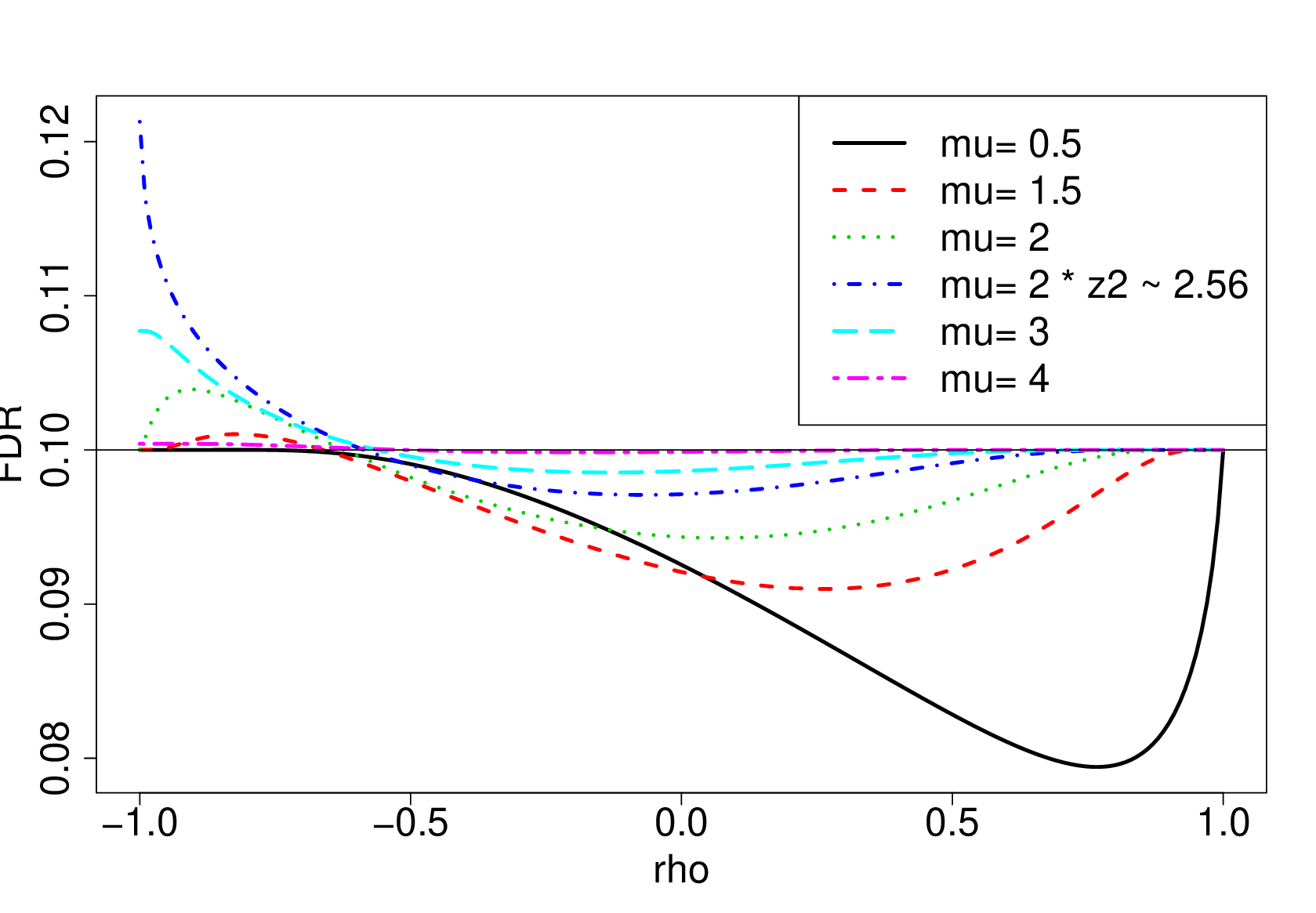}
\includegraphics[width=0.24\textwidth,angle=-90]{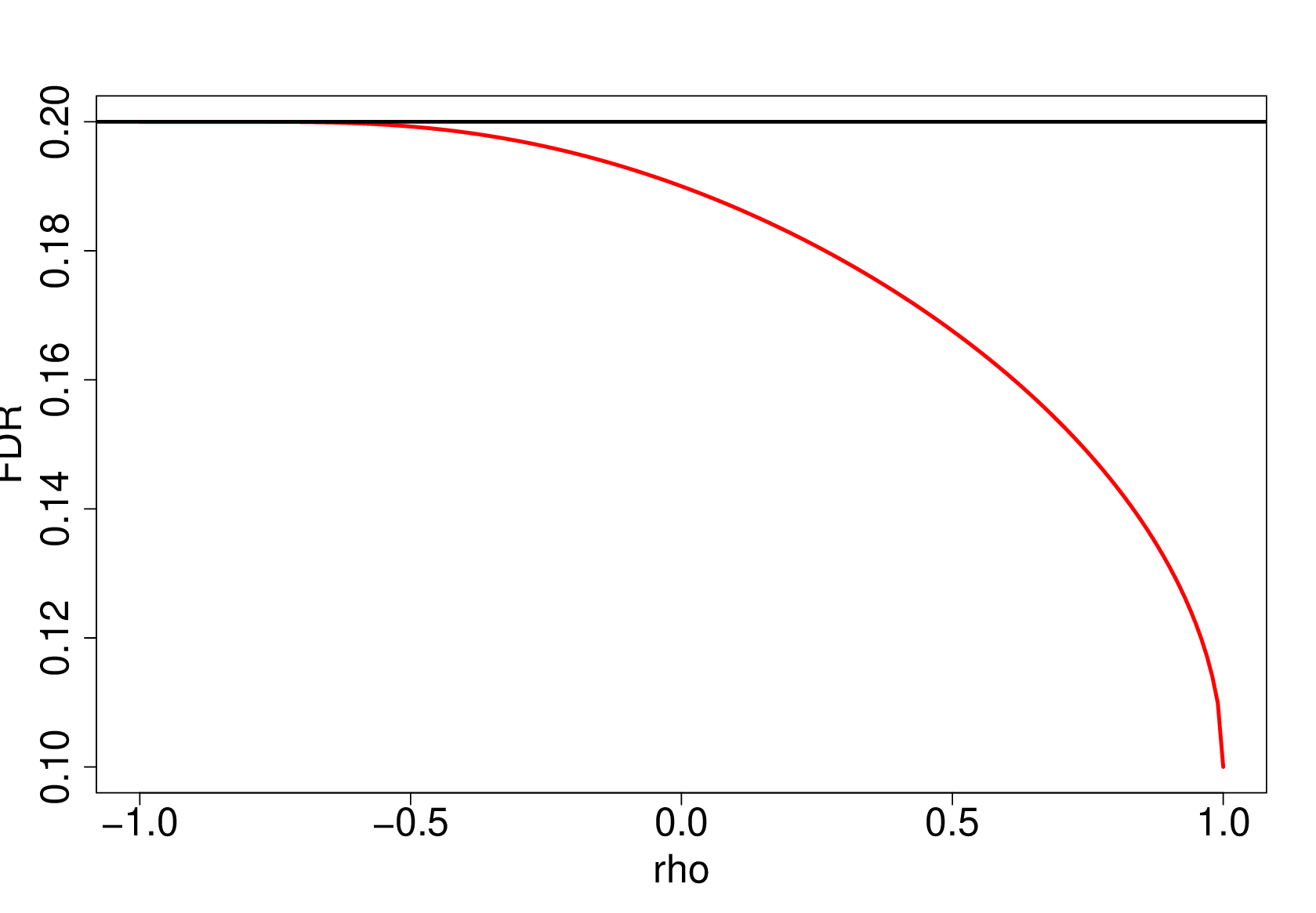}
\end{center}%\vspace{-5mm}
\caption{FDR against the covariance $-1\leq \rho \leq 1$ in the conditional EMN model. $m=2$; $\alpha=0.2$. Left: $m_0=1$, right: $m_0=2$. \label{negcorr}Top: LSU, bottom: LSD.}
\end{figure}

\begin{remark}
Reiner-Benaim (2007) also studied the value of $\mu$ maximizing the FDR in the case $m=2$ in the (conditional) EMN model with possibly negative correlation \cite{Reiner2007}. The latter work focused on the two-sided testing with $\rho\in\{-1,1\}$, $m_0=1$ and $m=2$. %Our improvement with respect to this work is therefore mainly to deal with $\rho$ varying in all the interval $[-1,1]$ and not only in $\{-1,1\}$.
\end{remark}

\subsection{Least/most favorable configurations for the variance of the FDP}\label{sec:LFCvar}

We focus here on the unconditional independent model and on the LSU procedure.
Using \eqref{equ_mom_FDP_indep} with $s=2$, we easily derive the following expression for the variance of the FDP:  
\begin{align}
\V[\FDP(\LSU) ]  &= \alpha \pi_0 \sum_{k=1}^{m} \frac{1}{k}  \distsu_{m-1}\big(  [G(\alpha (j+1)/m)]_{1\leq j\leq m-1} ,k-1\big) - (\alpha\pi_0)^2 /m.\label{equ_VarFDP}
\end{align}
As a consequence, by contrast with the FDR which is constantly equal to $\pi_0 \alpha$ in that case, the variance of the FDP depends on the alternative $p$-value c.d.f. $F_1$. Moreover, since the sum in \eqref{equ_VarFDP} equals 
$\E_{0}[(|\SU([G(\alpha (j+1)/m)]_{1\leq j\leq m-1})|+1)^{-1} ]$ (where $\E_0$ denotes the expectation with respect to $m-1$ i.i.d. uniform $p$-values), the smaller $F_1$ (point-wise), the larger the variance. Therefore, over the set $F_1\in\mathcal{F}$, the least  favorable configuration for the variance (that is, the configuration where the variance is the largest) is given by $F_1=0$ while the most favorable configuration (that is, the configuration where the variance is the smallest) is the Dirac-uniform configuration $F_1=1$. Over the more ``realistic'' c.d.f. sets $\mathcal{F}'=\{F_1 \in \mathcal{F} \telque \forall x \in(0,1), F_1(x)\geq  x\}$ and $\mathcal{F}_{\varepsilon}=\{F_1 \in \mathcal{F} \telque \forall x \in(0,1), F_1(x)\geq \varepsilon\}$, $0<\varepsilon \leq 1$ the least  favorable configurations for the variance are given respectively by $F_1(x)=x$ and $F_1(x)=\varepsilon$.  
For these extreme configurations, the expression of the variance \eqref{equ_VarFDP} can be simplified by using the next formula (proved in Section~\ref{proofcor_LCPVAR}): for any threshold of the form $t_k=\beta+k\gamma$, $1\leq k \leq m$, with $\beta,\gamma\geq 0$, 
 \begin{equation}\label{key-rel-var}
% \E[ (|\SU(\mbf{t})|+1)^{-1} ] =  
\E_0[(|\SU(\mbf{t})|+1)^{-1} ] =  \frac{1}{\gamma-\beta}\bigg[\frac{(1+\gamma-\beta)^{m+1}-1}{m+1} - \gamma \big[(1+\gamma-\beta)^{m}-1\big]\bigg]
 \end{equation}
 for $\gamma\neq \beta$ and $\E_0[(|\SU(\mbf{t})|+1)^{-1} ]=1-m\gamma$ otherwise (where $\E_0$ denotes the expectation with respect to $m$ i.i.d. uniform $p$-values)  This leads to the following result.

\begin{theorem}\label{cor_LCPVAR}
Consider the linear step-up procedure LSU in the unconditional independent model with parameters $\pi_0$ and $F_1$. Then for any $m\geq 2$, $\alpha \in (0,1)$, $\pi_0\in [0,1]$ and $\varepsilon\in(0,1]$, under the generating distribution ${(H,\mathbf{p})\sim \PExI_{(\pi_0,F_1)}}$, the following holds:
\begin{align}
 \min_{F_1\in\mathcal{F}} \{\V[\FDP(\LSU) ] \} &= \min_{F_1\in\mathcal{F}'} \{\V[\FDP(\LSU) ] \} = \min_{F_1\in\mathcal{F}_\varepsilon} \V[\FDP(\LSU) ] \} \nonumber\\
 &= \frac{\alpha \pi_0}{m} \frac{1-\pi_0^m}{1-\pi_0} -  \frac{(\alpha \pi_0)^2}{m} \left(\frac{1-\pi_0^{m-1}}{1-\pi_0} +1\right)\label{equ_BFPVAR}\\
 \max_{F_1\in\mathcal{F}} \{\V[\FDP(\LSU) ] \} &=  \alpha\pi_0 (1-\alpha\pi_0)\nonumber\\
 \max_{F_1\in\mathcal{F}'} \{\V[\FDP(\LSU) ] \} &= \alpha\pi_0(1-\alpha) + (1-\pi_0)\frac{\pi_0 \alpha^2}{m} \nonumber\\
 \max_{F_1\in\mathcal{F}_{\varepsilon}} \{\V[\FDP(\LSU) ] \} &=  \frac{\alpha \pi_0}{m} \frac{1-(1-(1-\pi_0)\varepsilon)^m}{(1-\pi_0)\varepsilon} -  \frac{(\alpha \pi_0)^2}{m} \left(\frac{1-(1-(1-\pi_0)\varepsilon)^{m-1}}{(1-\pi_0)\varepsilon} +1\right). \nonumber
\end{align}
\end{theorem}

The proof is made in Section~\ref{proofcor_LCPVAR}. %For $\varepsilon=0$, that is $F=0$, $\FDP$ trivially has a Bernoulli distribution with parameter $\pi_0$, so that we can check that the variance is well given by the above result: $\alpha \pi_0(1-\alpha \pi_0)$.
Using Theorem~\ref{cor_LCPVAR}, we are able to investigate the following asymptotic multiple testing issue:  does the $\FDP$ converge to the $\FDR$ as $m$ grows? Establishing the latter is crucial because when one establishes $\FDR\leq \gamma$, one implicitly wants that for the observed realization $\omega$, the control $\FDP(\omega)\leq \gamma'$ still holds for $\gamma'\simeq \gamma$ (at least with high probability and when $m$ is large). 
Here, the variance measure the $L^2$ distance between the FDP and the FDR, and since $\FDP\in[0,1]$ the latter distance tends to zero if and only if the FDP converges to the FDR in probability.  First, if $\pi_0\in(0,1)$ does not depend on $m$, the convergence holds over the set of c.d.f.'s $\mathcal{F}_{\varepsilon}$, with a distance $(\E(\FDP - \FDR)^2)^{1/2}$ converging to zero at rate $1/\sqrt{m}$. This corroborates previous asymptotic studies in the so-called ``subcritical'' case (see e.g. \cite{Chi2007,Neu2008}). By contrast, when considering the  classes $\mathcal{F}$ and $\mathcal{F}'$ the convergence does not hold in the least favorable configurations $F_1(x)=0$ and $F_1(x)=x$, respectively. The latter is quite intuitive because the denominator in the FDP does not converge to infinity anymore in that cases (see e.g. \cite{FR2002}), so these configurations can probably be considered as ``marginal''. Second, our non-asymptotic approach allows to make $\pi_0$ depends on $m$ in the following way $1-\pi_0=1-\pi_{0,m}\sim m^{-\beta}$ with $0< \beta \leq 1$, which corresponds to a classical ``sparse'' setting (see e.g. \cite{DJ2004}). %This corresponds to many practical situations and thus is a case of special interest (see e.g. ). 
Expression \eqref{equ_BFPVAR} implies that, in this sparse case, the variance is always larger than a quantity of order $1/m^{1-\beta}$. In particular, when $1-\pi_{0,m}\sim 1/m$, for any alternative c.d.f. $F_1$, the FDP does not converge to the FDR, and when $1-\pi_{0,m}\sim m^{-\beta}$ with $0< \beta < 1$, for all $F_1$, the convergence  of the FDP towards the FDR is of order slower than $1/m^{(1-\beta)/2}$ (in $L^2$ norm). As illustration, for $m=10000$, $1-\pi_0=1/100$ and  $\alpha=0.05$, expression \eqref{equ_BFPVAR} gives $(\E(\FDP - \FDR)^2)^{1/2} \geq 0.0217$, so the FDP has a distribution quite spread around the $\FDR=\pi_0 \alpha \simeq 0.05$. % means nothing very precise for the location of the true FDP. 
As a conclusion, considering a sparse signal slows down the convergence of the FDP to the FDR, so any FDR control should be interpreted with cautious, even in this very standard framework (independent $p$-values with the LSU procedure). %This corroborates the discussion that we made in Section~\ref{sec:mainresults:stepup}.

 %careful when interpreting a FDR control interpretation  in a situation where $\pi_0$ tends rapidly to $1$.

%Next, for $\varepsilon=1$, we have that $\max_{F\in\mathcal{F}_1} \{\V[\FDP(\LSU,H) ] \}=   $ which does not converge to zero when $m$ grows to infinity
%$\max_{F\in\mathcal{F}} \{\V[\FDP(\LSU,H) ] \}=  \alpha\pi_0(1-\alpha \pi_0)$
%... so controlling the FDR is not always relevant  (this means nothing for the FDP !). 

%P(FDP<2alpha) petite donc pas juste autour de alpha

\section{Extensions}\label{sec:discuss}
\begin{comment}
Several extensions of our approach can be proposed. First, in the independent model, we may replace in our formulas $F_0(x)=x$ by any continuous c.d.f. This allows to deal with possibly non-uniform test statistics under the null (for instance, the latter was useful in the equicorrelated case, see Section~\ref{sec:deppossu}). %For instance, this can be useful to consider $p$-values coming from discrete test statistics, for which the case $F_0(x)<x$ is possible. 
Second, as described in Section~2 of \cite{FDR2007}, the device that we used in Section~\ref{sec:deppossu} to directly extend our formulas from the independent case to the equi-correlated normal case works for any suitable test statistic of the form $T_i=g(X_i,U)$ with independent variables $(X_i)_i$ and a variable $U$ independent from the $X_i$'s. Therefore, our methodology also applied in such a model.  
Third, 
\end{comment}
Our approach is also useful  to study the false non-discovery proportion (FNDP), that is, the proportion of false hypotheses among the non-rejected hypotheses, and in particular the false non-discovery rate (FNR), defined as the average  of the FNDP. For this, we use the following duality property between step-up and step-down procedures:  point-wise, the hypotheses rejected by $\SD(\mbf{t},\mbf{p})$ are exactly the hypotheses non-rejected by  $\SU(\ol{\mbf{t}},\ol{\mbf{p}})$ with $\ol{p}_i=1-p_i$ and $\ol{t}_r=1-t_{m-r+1}$. Hence the distribution of  the FNDP of a step-down procedure can be deduced from the distribution of the FDP of a step-up procedure. Precisely, for $0\leq k\leq m-1$, the property \eqref{equ_distrFDP} implies that the distribution of the erroneous non-rejection number  $|\cH_1(H)\cap (\SD(\mbf{t}))^c |$ conditionally on $| \SD(\mbf{t}) |=k$ is binomial with parameters $m-k$ and ${\pi_1(1- F_1(t_{k+1}))}/{(1-G(t_{k+1}))}$. In particular, this leads to
\begin{align}
\FNR(\SD(\mbf{t}),\PExI_{\pi_0,F_1})=m\pi_1 \sum_{k=0}^{m-1} \frac{1-F_1(t_{k+1})}{m-k} \distsd_{m-1}((G(t_j))_{1\leq j\leq m-1},k).\nonumber
\end{align}
Moreover, applying once more the duality property between step-up and step-down, we deduce from Section~\ref{sec:indep-step-down} that for a step-up procedure, the distribution of the erroneous non-rejection number conditional on the rejection number is not binomial, in general, while we can still obtain an explicit expression for the FNR. % rejecting $\hat{k}$ hypotheses, by computing the joint distribution between $\hat{k}_{(1)}$ and $\hat{k}'_{(1)}$ (with definitions analogue to those given in Lemma~\ref{lemma_down}). %, which can be computed using a methodology similar to the one used in the step-down case (see  \eqref{equ_jointdist_sd}). 
%As a matter of fact, it is interesting to note the following duality in our exact expressions: the FDR of a step-up (resp. step-down) procedure has a form similar to the FNR of a step down (resp. step-up) procedure. This is related to the following property: point-wise, the hypotheses rejected by $\SD(\mbf{t},\mbf{p})$ are exactly the hypotheses non-rejected by  $\SU(\ol{\mbf{t}},\ol{\mbf{p}})$ with $\ol{p}_i=1-p_i$ and $\ol{t}_r=1-t_{m-r+1}$.

Finally, since our formulas depend on the true parameters of the model, which are in general unknown in a statistical approach, one may formulate the concern of estimating these quantities in our formulas. We did not investigate in detail this issue as it would exceed the scope of this paper. Here, we simply notice that plugging convergent estimators of the parameters in our formulas will lead to convergent estimators for the corresponding quantity (e.g. $\P(\FDP\leq x)$, FDR or Power), %and the $s$-th moment of the FDP, 
because our formulas are continuous in all the model parameters.

%Finally, let us underline that the bounds for the variance of the FDP  obtained in Theorem~\ref{cor_LCPVAR} show that the FDR can be sometimes far to the real observed FDP value, even for a standard setting assuming independent $p$-values and considering the linear step-up procedure. This reinforces the interest in FDP controls of the form $\prob{ \FDP \leq \alpha}\geq 1-\gamma$, as suggested in \cite{GW2004,LR2005}. 
%An exact computation of the upper-tail of the distribution of the FDP may therefore be an interesting direction for future research.

\section{Proofs}

\subsection{Proof of Theorem~\ref{main_indep}}

Let us first prove \eqref{equ_distrFDP} by computing the joint distribution of $|\cH_0(H)\cap \SU(\mbf{t}) |$ and $| \SU(\mbf{t}) |$. In the independent unconditional model, we may use the exchangeability of $(H_i,p_i)_i$ to obtain for any $0\leq j \leq \l \leq m$,
\begin{align*}
&\P[|\cH_0(H)\cap \SU(\mbf{t}) | = j, | \SU(\mbf{t}) | = \l ]\\
&= {{\l} \choose {j}} {{m} \choose {\l}} \P[\cH_0(H)\cap \SU(\mbf{t})  = \{1,...,j\},  \SU(\mbf{t})  = \{1,...,\l\} ]\\
 &= {{\l} \choose {j}} {{m} \choose {\l}} \P[ \SU(\mbf{t})  = \{1,...,\l\} , H_1=...=H_j=0, H_{j+1}=...=H_{\l}=1].
\end{align*}
Next, by definition of a step-up procedure, we have $\SU(\mbf{t})=\{i\telque p_i\leq {t}_{\hat{k}}\}$ with $\hat{k}=|\SU(\mbf{t})|$ (expression related to the ``self-consistency''  condition introduced in \citep{BR2008EJS}). Using Lemma~\ref{lemma_stepup}, if $\hat{k}'_{(\l)}$ denotes the number of rejections of the step-up procedure of threshold $(t_{j+\l})_{1\leq j \leq m-\l}$ over $m-\l$ hypotheses and using the $p$-values $p_{\l+1}$, ..., $p_{m}$, we have
\begin{align*}
\SU(\mbf{t})  = \{1,...,\l\}  &\Longleftrightarrow p_1\leq t_\l, ... , p_\l \leq t_\l , \,\hat{k}  = \l\\
 &\Longleftrightarrow p_1\leq t_\l, ... , p_\l \leq t_\l ,\,\hat{k}'_{(\l)} = 0.
\end{align*}
Therefore, 
\begin{align*}
&\P[|\cH_0(H)\cap \SU(\mbf{t}) | = j, | \SU(\mbf{t}) | = \l ]\\
 &= {{\l} \choose {j}} {{m} \choose {\l}} \P[ p_1\leq t_\l, ... , p_\l \leq t_\l ,\,\hat{k}'_{(\l)} = 0, H_1=...=H_j=0, H_{j+1}=...=H_{\l}=1]\\
 &= {{\l} \choose {j}} {{m} \choose {\l}} \P[ p_1\leq t_\l, ... , p_j \leq t_\l , H_1=...=H_j=0] \\
 &\,\,\times\P[ p_{j+1}\leq t_{\l}, ... , p_\l \leq t_\l , H_{j+1}=...=H_{\l}=1]\, \P[ \hat{k}'_{(\l)} = 0]\\
 &=  {{\l} \choose {j}} {{m} \choose {\l}} (\pi_0F_0(t_\l))^j (\pi_1F_1(t_\l))^{\l-j}\Psi_{m-\l}\big(1-G(t_m),...,1-G(t_{\l+1})\big),
\end{align*}
where we used the independence between the $(H_i,p_i)$ in the second equality. This leads to \eqref{equ_distrFDP} and then to \eqref{equ_FDP_indep}. For  \eqref{equ_mom_FDP_indep}, we use the Stirling numbers of second kind %\eqref{equ_coeff_Bls}
 and the formula of the $s$-th moment of a binomial distribution of Section~\ref{sec:notation}. Expression \eqref{equ_FDR_indep} is a direct consequence of \eqref{equ_mom_FDP_indep} for $s=1$. 
For the power computation, from \eqref{equ_distrFDP}, the distribution of $|\cH_1(H)\cap \SU(\mbf{t}) |$ conditionally on $\hat{k}=| \SU(\mbf{t}) |$
is binomial with parameters $\hat{k}$ and $\pi_1 F_1(t_{\hat{k}})/G(t_{\hat{k}})$. Therefore,
$\E[|\cH_1(H)\cap \SU(\mbf{t}) |]=   \E[\pi_1 \hat{k} F_1(t_{\hat{k}})/G(t_{\hat{k}})] $ and \eqref{equ_Pow_indep} follows.

\subsection{Proof of Theorem~\ref{main_indep_sd}}\label{proof_equ_jointdist_sd}

Let us prove the FDR expression (the proof for the power is similar). Define $\tilde{k}=|{\SD(\mbf{t})}|$ and $\tilde{k}_{(1)}$, $\tilde{k}'_{(1)}$ as in Lemma~\ref{lemma_down}. We get by exchangeability of $(H_i,p_i)_i$ and  independence of the $p$-values,
\begin{align*}
\FDR(R) &=\sum_{i=1}^m  \E\bigg[\frac{\ind{p_i\leq t_{\tilde{k}}}}{\tilde{k}\vee 1} \ind{H_i=0}\bigg] = m  \E\bigg[\frac{\ind{p_1\leq t_{\tilde{k}}}}{\tilde{k}\vee 1} \ind{H_1=0}\bigg]\\
&= m   \E\bigg[\frac{\ind{p_1\leq t_{\tilde{k}_{(1)}+1}}}{\tilde{k}'_{(1)}+1}\ind{H_1=0}\bigg] =\pi_0 m   \E\bigg[\frac{ F_0(t_{\tilde{k}_{(1)}+1})}{\tilde{k}'_{(1)}+1}\bigg].
\end{align*}
Therefore, expression \eqref{equ_FDR_indep_sd} will be proved as soon as we state that
for any $k,k'$ with $0\leq k \leq m$ and $k\leq  k' \leq m$, that we have
\begin{align}
& \prob{|\SD(\mbf{t})|=k,|\SD(\mbf{t}')|=k' }\nonumber \\&=\: \distsd_m\big((G(t_j))_{1\leq j\leq m},k\big)\: \distsd_{m-k} \left( \left(\frac{G(t_{k+1+j})-G(t_{k+1})}{1-G(t_{k+1})}\right)_{1\leq j\leq m-k},\:k'-k \right), \label{equ_jointdist_sd}
\end{align}
for any threshold $(t_j)_{1\leq j\leq m+1 }$ and letting $\mbf{t}=(t_j)_{1\leq j \leq m}$ and $\mbf{t}'=(t_{j+1})_{1\leq j \leq m}$  .
To prove \eqref{equ_jointdist_sd}, remark that we may assumed that $G(x)=x$ up to replace $(t_j)_j$ by $(G(t_j))_j$ (because $G$ is continuous increasing).
Next, assume $k<k'<m$ and denote $L(r)=\sum_{i=1}^m \ind{p_i\leq t_r}$. By definition of a step-down procedure, the probability $\prob{|\SD(\mbf{t})|=k,|\SD(\mbf{t}')|=k'}$ is equal to
\begin{align*}
&\hspace{-1cm}\prob{ \forall j \leq k, L(j)\geq j , L(k+1)=k,  \forall j, k+1\leq j \leq k', L(j+1)\geq j, L(k'+2)=k'}\\
=&\: {m \choose k}  {m-k \choose k'-k} \:\:\:\P\big[ \forall j \leq k, L(j)\geq j , \forall j, k+1\leq j \leq k', L(j+1)\geq j, \\
&p_1,...,p_k \leq t_{k+1} < p_{k+1}, ..., p_{k'} \leq t_{k'+2} < p_{k'+1}, ...,  p_m\big]\\
=& \:{m \choose k}  {m-k \choose k'-k} \:\:\:\P\bigg[ p_1,...,p_k \leq t_{k+1} < p_{k+1}, ..., p_{k'} \leq t_{k'+2} < p_{k'+1}, ...,  p_m\,, \\
&\forall j \leq k, \sum_{i=1}^k \ind{p_i\leq t_j}\geq j , \:\:\forall j, k+1\leq j \leq k',  \sum_{i=k+1}^{k'} \ind{p_i\leq t_{j+1}} \geq j-k  \bigg]\\
=&\: {m \choose k}  {m-k \choose k'-k} \prob{ p_{(1)} \leq   t_1,...,p_{(k)}\leq t_k}(1-t_{k'+2})^{m-k'}\\
&\times \prob{ t_{k+1}<p_{(1)} \leq  t_{k+2},...,p_{(k'-k)}\leq t_{k'+1}},
\end{align*}
where we used that the $p$-values are i.i.d. Simple computations give that
$$
\prob{ t_{k+1}<p_{(1)} \leq  t_{k+2},...,p_{(k'-k)}\leq t_{k'+1}} = \prob{ p_{(1)} \leq  t_{k+2}-t_{k+1},...,p_{(k'-k)}\leq t_{k'+1}-t_{k+1}}
.$$
This leads to \eqref{equ_jointdist_sd}. The cases $k<k'=m$ and $k=k'$ are similar.

\subsection{Proof of \eqref{equ_distr_SD_1}, \eqref{equ_distr_SD_2} and \eqref{equ_distr_SD_3}}\label{preuve_distr_SD}

Let us state the following general expression: for $1\leq j\leq k \leq m$,
\begin{align}
&\P[|\cH_0(H)\cap \SD(\mbf{t}) |=j, |\SD(\mbf{t})|=k] \nonumber\\
&= {{k} \choose {j}}\pi_0^{j} \pi_1^{k-j} \frac{\P[U_{(1)}\leq t_{1},..., U_{(k)}\leq t_{k}\telque H_1=...=H_j=0, H_{j+1}=...=H_k=1]}{\Psi_{k}\big(G(t_1),...,G(t_k))},\label{equ_distr_SD}
\end{align}
where $(H_i,U_i)_{1\leq i \leq k}$ is a sequence of i.i.d. variables following the unconditional independent model (for $k$ hypotheses).
Then, applying \eqref{equ_distr_SD} in the case $j=k$, $j=0$ and $F_1=1$ will lead to  \eqref{equ_distr_SD_1}, \eqref{equ_distr_SD_2} and \eqref{equ_distr_SD_3}, respectively. To state \eqref{equ_distr_SD}, we use that $(H_i,p_i)_i$ is a i.i.d. sequence: 
\begin{align*}
&\P[|\cH_0(H)\cap \SD(\mbf{t}) |=j , |\SD(\mbf{t})|=k] \\
&= {k \choose j}   {m \choose k}  \P[  \SD(\mbf{t}) =\{1,...,k\}, H_1=...=H_j=0, H_{j+1}=...=H_k=1]\\
&= {k \choose j}   {m \choose k}  \P[ \forall \l \leq k, \sum_{i=1}^k  \ind{p_i\leq t_\l} \geq \l , \forall i\geq k+1, p_i>t_{k+1}, \\
&\:\: H_1=...=H_j=0, H_{j+1}=...=H_k=1]\\
&= {{k} \choose {j}}\pi_0^{j} \pi_1^{k-j} \P[ \forall \l \leq k, \sum_{i=1}^k  \ind{p_i\leq t_\l} \geq \l  \telque H_1=...=H_j=0, H_{j+1}=...=H_k=1] \\
&\:\:\times  {m \choose k}  (1-G(t_{k+1}))^{m-k},
\end{align*}
which leads to \eqref{equ_distr_SD} by applying \eqref{equ_dist_sd}.

\begin{comment}
Let us state \eqref{equ_distr_SD_1}  (the proof for \eqref{equ_distr_SD_2} is similar). Denoting $L(r)=\sum_{i=1}^m \ind{p_i\leq t_r}$ and using that $(H_i,p_i)_i$ are i.i.d. in the unconditional model, we obtain for $k\geq 1$, $k\leq m$,
\begin{align*}
&\P[|\cH_0(H)\cap \SD(\mbf{t}) |=k , |\SD(\mbf{t})|=k] \\
&=  \prob{ \forall j \leq k, L(j)\geq j , L(k+1)=k \mbox{ and $\forall i = 1, ... ,m$ s.t. $p_i \leq t_k$, $H_i=0$}  }\\
&= {m \choose k} \prob{ \forall j \leq k, \sum_{i=1}^k \ind{p_i\leq t_j} \geq j , H_1=...=H_k=0, \forall i >k, p_i>t_{k+1} }\\
&= {m \choose k} \prob{ \forall j \leq k, \sum_{i=1}^k \ind{p_i\leq t_j} \geq j , H_1=...=H_k=0} (1-G(t_{k+1}))^{m-k}\\
&= {m \choose k} \pi_0^k \Psi_k( F_0(t_1), ..., F_0(t_k)) (1-G(t_{k+1}))^{m-k}.
\end{align*}
We derive \eqref{equ_distr_SD_1}  by dividing the previous expression by $\distsd_m\big((G(t_j))_{1\leq j\leq m},k\big)$.
\end{comment}

\subsection{Proof of Theorem~\ref{th_m2}}

We focus on the step-up case, when $\rho\in(-1,1)$ and $m_0=1$ (the remaining cases are left to the reader). Without loss of generality, we may assume that the first coordinate correspond to the true null, that is, $H=(0,1)$. In this context, the FDP takes one of the three values: $0, \frac{1}{2}, 1$, according to the location of the tests statistics $Y_i=X_i+\mu H_i$ with respect to the critical values $z_1$ and $z_2$. From the definition of a step-up procedure, we may define the two regions for $i \in \{1,2\}$, $\mathcal{D}_i=\{(y_1,y_2)\in\mathbb{R}^2 \:|\: \FDP(y_1,y_2)=i/2 \}$, where  $\FDP(y_1,y_2)$ denotes the FDP of $\SU(\mbf{t})$ taken in the $p$-values $\mbf{p}=(\overline{\Phi}(y_1),\overline{\Phi}(y_2))$. The regions $\mathcal{D}_i$ are represented on Figure \ref{m01}({\bf{a}}). Next, since $(Y_1,Y_2)$ follows the EMN model, 
we may write for $i \in \{1,2\}$, $\P(\FDP(Y_1,Y_2)={i}/{2})=(2\pi  \sqrt{1-\rho^2})^{-1}  \int \int_{\mathcal{D}_i}
\exp\{-\frac{1}{2(1-\rho^2)} \big(y_1 -\rho (y_2-\mu)\big)^2 -  \frac{1}{2} (y_2-\mu)^2   \}\,dy_1dy_2=(2\pi)^{-1}  \int \int_{\mathcal{\tilde{D}}_i}
\exp\{-\frac{1}{2} (u^2 +v^2)\} \,du dv$
by using the substitution $u =(\sqrt{1-\rho^2})^{-1} \big(y_1 -\rho
(y_2-\mu)\big)$ and $v=y_2-\mu$, and where the resulting integration domain  $\mathcal{\tilde{D}}_i$ is represented on Figure \ref{m01}({\bf{b}}). Therefore, we  obtain $
\P(\FDP={1}/{2})= ({\sqrt{2\pi}})^{-1} \int_{z_1-\mu}^{\infty} {\exp\{- v^2/2\}} \ol \Phi\left({(z_1-\rho v)}/{\sqrt{1-\rho^2}}\right)\,dv, $ and 
$\P(\FDP=1)= ({\sqrt{2\pi}})^{-1}  \int_{-\infty}^{z_1-\mu} {\exp\{-
    v^2/2\}}\ol \Phi\left({(z_2-\rho v)}/{\sqrt{1-\rho^2}}\right)\,dv,
$
and the final expression results by using the substitution $w=\ol \Phi(v)$.

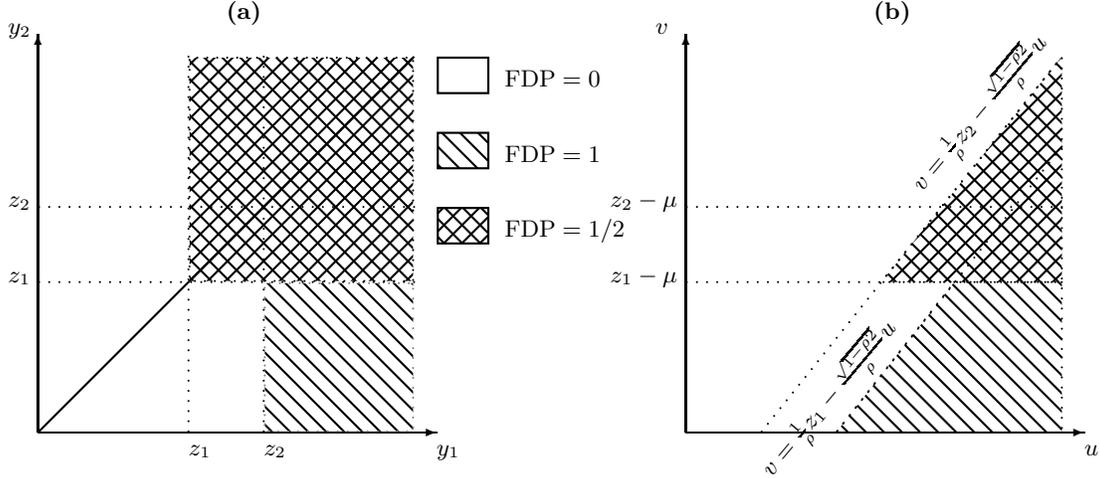
\begin{figure}[hbtp]
\vspace{0.5cm}
\hspace{-0.7cm}\begin{minipage}{5cm}\begin{pspicture}(5,5)
\put(0,0){\vector(1,0){5.3}}
\put(0,0){\vector(0,1){5.3}}
\put(5.3,-0.3){$y_1$}
\put(-0.4,5.3){$y_2$}
\psline(0,0)(5,5)
\psline[linestyle=dotted](2,0)(2,5)
\psline[linestyle=dotted](3,0)(3,5)
\psline[linestyle=dotted](0,2)(5,2)
\psline[linestyle=dotted](0,3)(5,3)
\put(2,-0.3){$z_1$}
\put(3,-0.3){$z_2$}
\put(-0.4,2){$z_1$}
\put(-0.4,3){$z_2$}
\psframe[linestyle=dotted,fillstyle=vlines,linecolor=gray]%
(3,0)(5,2)
\psframe[linestyle=dotted,fillstyle=crosshatch]%
(2,2)(5,5)
\psframe[fillstyle=vlines]%
(5.3,3.5)(6,4)
\psframe[fillstyle=crosshatch]%
(5.3,2.5)(6,3)
\psframe%
(5.3,4.5)(6,5)
\put(6.2,3.6){$\FDP=1$}
\put(6.2,2.6){$\FDP=1/2$}
\put(6.2,4.6){$\FDP=0$}
\put(2.5,5.5){\bf{(a)}}
\end{pspicture}
\end{minipage}
\hspace{3.5cm}\begin{minipage}{4cm}
\begin{pspicture}(5,5)
\put(0,0){\vector(1,0){5.3}}
\put(0,0){\vector(0,1){5.3}}
\put(5.3,-0.3){$u$}
\put(-0.4,5.3){$v$}
\psline[linestyle=dotted](2,0)(5,3.8)
\psline[linestyle=dotted](1,0)(5,5)
\psline[linestyle=dotted](0,2)(5,2)
\psline[linestyle=dotted](0,3)(5,3)
\put(-1,2){$z_1-\mu$}
\put(-1,3){$z_2-\mu$}
%\put(3.6,5.2)
\put(2.9,3.2)
{\begin{turn}{50} $ v=\frac{1}{\rho} z_2 -\frac{\sqrt{1-\rho2}}{\rho} u$ \end{turn}}
\put(0.9,-0.6){\begin{turn}{50} $ v=\frac{1}{\rho} z_1 -\frac{\sqrt{1-\rho2}}{\rho} u$ \end{turn}}
\pspolygon[linestyle=dotted,fillstyle=vlines]%
(5,0)(2,0)(3.57,2)(5,2)
\pspolygon[linestyle=dotted,fillstyle=crosshatch]%
(5,5)(2.6,2)(5,2)
\put(2.5,5.5){\bf{(b)}}
\end{pspicture}
\end{minipage}
\vspace{0.3cm}
\caption{Left: $\mathcal{D}_i$. Right: $\mathcal{\widetilde{D}}_i$. For $i=1$(double cross-hatch) and $i=2$ (simple cross-hatch). % after the variable  substitution  $u =\frac{1}{\sqrt{1-\rho2}} \big(x_1 -\rho (x_2-\mu)\big)$ and $v=x_2-\mu$ 
Graph for $\rho<0$.} 
\label{m01}
\end{figure}

\subsection{Proof of Theorem~\ref{appli_indep_sd} }\label{proof_appli_indep_sd}

For any $t\in[0,1]$, let $G(t)=\pi_0 t + \pi_1 F_1(t)$ and $G_1(t)=\pi_0 t + \pi_1$. 
First, since $F_1$ is concave,  we have for $t<t'\leq 1$, that $(F_1(t')-F_1(t))/(t'-t)\geq (1-F_1(t))/(1-t)$ and thus for $t\leq t'$, we obtain the inequality
\begin{equation}(G(t')-G(t))/(1-G(t)) \geq (G_1(t')-G_1(t))/(1-G_1(t))\label{equ_H}\end{equation}
%We denote $\underline{\mbf{t}}=(t_j)_{1\leq j \leq m-1}$,  $\ol{\mbf{t}}=(\alpha j/m)_{2\leq j \leq m}$,  $G\circ\underline{\mbf{t}}=(G(\alpha j/m))_{1\leq j \leq m-1}$ and $G\circ\overline{\mbf{t}}=(G(\alpha j/m))_{2\leq j \leq m}$.  
(by convention, the LHS (resp. RHS) of \eqref{equ_H} is equal to $0$ if $G(t)=1$ (resp. $G_1(t)=1$)). From expression \eqref{equ_FDR_indep_sd},  we obtain that $\FDR(\LSD,\PExI_{(\pi_0,F)})$ equals
\begin{align*}
 &\pi_0  m \sum_{k=1}^{m}    \distsd_{m-1}((G(t_j))_{1\leq j \leq m-1},k-1)\sum_{i=0}^{m-k}  \frac{t_k}{k+i} 
\distsd_{m-k} \left( \left(\frac{G(t_{k+j})-G(t_{k})}{1-G(t_{k})}\right)_{1\leq j\leq m-k},i \right)\\
&\leq \pi_0  m \sum_{k=1}^{m}   \distsd_{m-1}((G(t_j))_{1\leq j \leq m-1}, k-1) \sum_{i=0}^{m-k} \frac{t_k}{k+i} 
 \distsd_{m-k} \left( \left(\frac{G_1(t_{k+j})-G_1(t_{k})}{1-G_1(t_{k})}\right)_{1\leq j\leq m-k},i \right)\\
&= \pi_0  m \sum_{k=1}^{m}  \distsd_{m-1}((G(t_j))_{1\leq j \leq m-1},k-1) \sum_{i=0}^{m-k} \frac{t_k}{k+i} 
\distsd_{m-k} \left( \left(\frac{t_{k+j}-t_k}{1-t_k}\right)_{1\leq j\leq m-k},i \right),
\end{align*}
where the inequality comes from  \eqref{equ_H} and because for a fixed
$k$, the sum over $i$ can be seen as the expectation of $t_k/(k+I)$
where $I$ is the rejection number of a step-down procedure (point-wise
nondecreasing in the threshold). Next, considering this time the sum
over $k$ as an expectation, since $G\leq G_1$ and since a step-down
procedure is point-wise nondecreasing in the threshold, the proof is
finished by using assumption \eqref{fun_croiss}.

Consider now the case where $t_j=\alpha j/m$ and let us prove that this threshold satisfies  \eqref{fun_croiss}. For any $m\geq 2$ and $1\leq k\leq m$, let us denote $S_{m,k}= \sum_{i=0}^{m-k} \frac{k}{k+i} \distsd_{m-k} \left(
  \left(\frac{\alpha j/m}{1-\alpha k/m}\right)_{1\leq j\leq m-k},i
\right)$. Letting $a_{m,k}=\frac{k}{m}\bigg(1- \alpha\: \frac{m-k}{m-\alpha k}\bigg)$ (increasing in $k$) and $b_{m,k}=\frac{m-\alpha}{m}\frac{m-k}{m-\alpha k}$ (decreasing in $k$),
we may prove the following recursion (see the proof below): 
\begin{equation}\label{equ_recursion}
S_{m,k}=a_{m,k}+b_{m,k}  S_{m-1,k}. 
\end{equation}
Using \eqref{equ_recursion} we propose to state that $(t_k)_k$ satisfies \eqref{fun_croiss}, that is,``$k\in\{1,...,m\}\mapsto S_{m,k}$ is nondecreasing'' by a recurrence on $m\geq 2$: the property is obviously true for $m=2$. Assuming the property true for $m-1$, % $S_{m-1,k}\geq S_{m-1,k-1}$, 
we obtain for any $2\leq k\leq m-1$,
\begin{align*}
S_{m,k}-S_{m,k-1} &= a_{m,k} -a_{m,k-1} +b_{m,k}  S_{m-1,k}-b_{m,k-1}  S_{m-1,k-1}\\
&= a_{m,k} -a_{m,k-1} +(b_{m,k} -b_{m,k-1} ) S_{m-1,k}+b_{m,k-1}  (S_{m-1,k}-S_{m-1,k-1})\\
&\geq  (a_{m,k} +b_{m,k}) - (a_{m,k-1} +b_{m,k-1} ),
\end{align*}
because $S_{m-1,k}\leq 1$. Hence, since $a_{m,k} +b_{m,k}=1- \frac{\alpha}{m} \frac{m-k}{m-\alpha k}$, the quantity $a_{m,k} +b_{m,k}$ is increasing in $k$ and $S_{m,k}-S_{m,k-1} \geq 0$. Also, we obviously have $S_{m,m}=1\geq S_{m,m-1}$, and the recurrence is completed.

We now finally state \eqref{equ_recursion}. Let for $1\leq j\leq m-k$, $t_j=\frac{\alpha j/m}{1-\alpha k/m}$ and $t'_j=1-t_{m-k-j+1}$, so that
\begin{align*}
S_{m,k}=\E_0\left[\frac{k}{k+ |\SD((t_{j})_{1\leq j \leq m-k})|} \right] = \E_0\left[\frac{k}{m-|\SU((t'_{j})_{1\leq j \leq m-k})|} \right],
\end{align*}
where $\E_0$ denotes the expectation with respect to i.i.d. uniform $p$-values.
Hence, denoting $t'_j=\beta +j \gamma$ with $\beta=\frac{m-m \alpha -\alpha}{m- \alpha k}$ and $\gamma=\frac{\alpha}{m- \alpha k}$, we obtain
\begin{align*}
\frac{S_{m,k}}{k}
=&\frac{1}{m}  + \frac{1}{m} \E_0\left[ \frac{|\SU((t'_{j})_{1\leq j \leq m-k})|}{m-|\SU((t'_{j})_{i\leq j \leq m-k})|} \right]\\
=&\frac{1}{m}+\frac{1}{m} \sum_{j=1}^{m-k} \frac{j}{m-j} {m-k \choose j}
(t'_{j})^j \Psi_{m-k-j}\big( 1-t'_{m-k},...,1-t'_{j+1}\big)\\
=&\frac{1}{m}+\frac{m-k}{m} \sum_{j=1}^{m-k} \frac{t'_j}{m-j} {m-k-1 \choose j-1}
(t'_{j})^{ j-1} \Psi_{m-k-j}\big( 1-t'_{m-k},...,1-t'_{j+1}\big),
 \end{align*}
so that $\frac{S_{m,k}}{k}$ equals
\begin{align*}
&\frac{1}{m}-\gamma \frac{m-k}{m}+ \frac{m-k}{m}(\beta+m \gamma)  \sum_{j=0}^{m-1-k} \frac{(t'_{j+1})^{ j}}{m-1-j} {m-1-k \choose j}
 \Psi_{m-1-k-j}\big( 1-t'_{m-k},...,1-t'_{j+2}\big) \\
&=\frac{1}{m}-\gamma \frac{m-k}{m}+ \frac{m-k}{m}(\beta+m \gamma)  \E_0\left[ \frac{1}{m-1-|\SU((t'_{j+1})_{1\leq j \leq m-1-k})|} \right] \\
&=\frac{1}{m}-\gamma \frac{m-k}{m}+ \frac{m-k}{m}(\beta+m \gamma) \frac{S_{m-1,k}}{k},
 \end{align*}
and the recursion \eqref{equ_recursion} is proved.

\subsection{Proofs for Section~\ref{sec:LFCvar}}\label{proofcor_LCPVAR}

Let us first prove \eqref{key-rel-var}:  denote for any $0\leq i \leq m-1$, $u_i= \E_0[(|\SU((t_{j})_{i+1\leq j \leq m})|+i+1)^{-1} ]$ (where $\E_0$ denotes the expectation with respect to i.i.d. uniform $p$-values) and $u_{m}=1/(m+1)$, $u_{m+1}=0$, so that $u_0$ equals the quantity in \eqref{key-rel-var}. We may prove the following recursion relation: for any $0\leq i \leq m$,
 \begin{equation}\label{recursion-rel-var}
(i+1)u_i = (1-(m-i)\gamma) - (m-i)(\beta-\gamma) u_{i+1}.
 \end{equation}
Expression \eqref{recursion-rel-var} is proved as follows: for $i<m$,
 \begin{align*}
%\E\left[\frac{i+1}{|\SU((t_{j})_{i+1\leq j \leq m})|+i+1} \right] 
(i+1)u_i=& 1- \E_0\left[ \frac{|\SU((t_{j})_{i+1\leq j \leq m})|}{|\SU((t_{j})_{i+1\leq j \leq m})|+i+1} \right] \\
=&1-\sum_{k=1}^{m-i} \frac{k}{k+i+1} {m-i \choose k} (t_{k+i})^k \Psi_{m-i-k}\big( 1-t_m,...,1-t_{k+i+1}\big)\\
=& 1- (m-i)\sum_{k=1}^{m-i} \frac{t_{k+i}}{k+i+1} {m-i-1 \choose k-1} (t_{k+i})^{k-1} \Psi_{m-i-k}\big( 1-t_m,...,1-t_{k+i+1}\big)\\
 =& 1- (m-i)\sum_{k=0}^{m-(i+1)} \left(\gamma+\frac{\beta-\gamma}{k+(i+1)+1}\right) {m-(i+1) \choose k} (t_{k+(i+1)})^{k} \\
&\times \Psi_{m-(i+1)-k}\big( 1-t_m,...,1-t_{k+(i+1)+1}\big). 
 \end{align*}
Next, we obtain that the solution of the recursion \eqref{recursion-rel-var} is given by
$$u_i = \sum_{j=0}^{m-i} \frac{1-(m-(i+j))\gamma}{m-(i+j)}  \:\:\frac{(m-i)\times \cdots \times (m-(i+j))}{(i+1) \times \cdots \times (i+j+1)} (\gamma-\beta)^j,$$
 which leads to $u_0 = \sum_{j=0}^{m} \frac{1-(m-j)\gamma}{j+1}{m \choose j} (\gamma-\beta)^j   $ and \eqref{key-rel-var} results.
 
To prove Theorem~\ref{cor_LCPVAR},  we combine  \eqref{equ_VarFDP} and \eqref{key-rel-var}, the latter using $m-1$ hypotheses and special values for $\beta$ and $\gamma$: $\beta=\gamma=\pi_0\alpha /m$ for  $F_1(x)=0$\,;  $\beta=\gamma=\alpha /m$ for  $F_1(x)=x$\,; $\beta=\pi_0\alpha /m+(1-\pi_0)\varepsilon$ and $\gamma=\pi_0\alpha /m$ for $F_1(x)= \varepsilon$.
%

%

%first note that since the FDR of the LSU procedure does not depend on $F$ and since $\E [\FDP2]$ is the largest when $F(x)= (\varepsilon\: x) \wedge 1$, it is sufficient to compute the variance for the latter alternative c.d.f.

%Next, from Theorem~\ref{main_indep}, we get the following formula for the variance of the FDP of the LSU in the independent case:
%\begin{align}
%\V[\FDP(\LSU,H) ]  &= \alpha \pi_0 \sum_{k=1}^{m} \frac{1}{k}  \distsu_{m-1}\big(  [G(t_{j+1})]_{1\leq j\leq m-1} ,k-1\big) - (\alpha\pi_0)2 /m\nonumber\\
%&=\alpha \pi_0  - \alpha \pi_0 \sum_{k=2}^{m} \frac{k-1}{k}   \distsu_{m-1}\big(  [G(t_{j+1})]_{1\leq j\leq m-1} ,k-1\big) - (\alpha\pi_0)2 /m\nonumber %{{m-1} \choose {k-1}}  (G( \alpha k/m))^{k-1} \Psi_{m-k}\big(1-G(\alpha m/m),...,1-G(\alpha (k+1)/m)\big)\label{equ_var_FDP_LSU_indep}.
%\end{align}
%Taking now $F(x)= (\varepsilon\: x) \wedge 1$, that is, $G(t_{j})= c j$ with $c= (\pi_0+ (1-\pi_0)\varepsilon)\alpha /m$, we obtain from %we may use \eqref{} and find that $\Psi_{m-k}\big(1-G(\alpha m/m),...,1-G(\alpha (k+1)/m)\big)- (\alpha\pi_0)2 /m=(1-c \alpha)(1-c \alpha k/m)^{m-k-1}$. This implies:
%$$
%\frac{k-1}{k}   \distsu_{m-1}\big(  [c(j+1)]_{1\leq j\leq m-1} ,k-1\big) =c (m-1)  \distsu_{m-2}\big(  [c(j+2)]_{1\leq j\leq m-2} ,k-2\big)
%$$
%that
%$
%\V[\FDP(\LSU,H) ] %\nonumber\\
%=\alpha \pi_0 - (\alpha\pi_0)2 /m - \alpha \pi_0 c(m-1)
%%&= \alpha \pi_0 - (\alpha\pi_0)2 /m -   c\pi_0 \alpha2  (1-1/m)= \alpha \pi_0(1-c\alpha) - \alpha2 \pi_0 \frac{\pi_0-c}{m}.
%$ and we can conclude.

\section*{Acknowledgements}

We would like to warmly acknowledge Gilles Blanchard, Sylvain Delattre and Zhan Shi for helpful discussions.
This work was supported by the French Agence Nationale de la Recherche (ANR grant references: ANR-09-JCJC-0027-01, ANR-PARCIMONIE, ANR-09-JCJC-0101-01).

\bibliography{biblio}

\begin{thebibliography}{}

\bibitem[Benjamini and Hochberg, 1995]{BH1995}
Benjamini, Y. and Hochberg, Y. (1995).
\newblock Controlling the false discovery rate: a practical and powerful
  approach to multiple testing.
\newblock {\em J. Roy. Statist. Soc. Ser. B}, 57(1):289--300.

\bibitem[Benjamini et~al., 2006]{BKY2006}
Benjamini, Y., Krieger, A.~M., and Yekutieli, D. (2006).
\newblock Adaptive linear step-up procedures that control the false discovery
  rate.
\newblock {\em Biometrika}, 93(3):491--507.

\bibitem[Benjamini and Yekutieli, 2001]{BY2001}
Benjamini, Y. and Yekutieli, D. (2001).
\newblock The control of the false discovery rate in multiple testing under
  dependency.
\newblock {\em Ann. Statist.}, 29(4):1165--1188.

\bibitem[Blanchard and Roquain, 2008]{BR2008EJS}
Blanchard, G. and Roquain, E. (2008).
\newblock Two simple sufficient conditions for {FDR} control.
\newblock {\em Electron. J. Stat.}, 2:963--992.

\bibitem[Blanchard and Roquain, 2009]{BR2008b}
Blanchard, G. and Roquain, E. (2009).
\newblock Adaptive {FDR} control under independence and dependence.
\newblock {\em J. Mach. Learn. Res.}, 10:2837--2871.

\bibitem[Chi, 2007]{Chi2007}
Chi, Z. (2007).
\newblock On the performance of {FDR} control: constraints and a partial
  solution.
\newblock {\em Ann. Statist.}, 35(4):1409--1431.

\bibitem[Chi and Tan, 2008]{Chi2008}
Chi, Z. and Tan, Z. (2008).
\newblock Positive false discovery proportions: intrinsic bounds and adaptive
  control.
\newblock {\em Statist. Sinica}, 18(3):837--860.

\bibitem[Dickhaus, 2008]{Dick2008}
Dickhaus, T. (2008).
\newblock {\em False Discovery Rate and Asymptotics}.
\newblock PhD thesis, Heinrich-Heine-Universit{\"a}t D{\"u}sseldorf.

\bibitem[Donoho and Jin, 2004]{DJ2004}
Donoho, D. and Jin, J. (2004).
\newblock Higher criticism for detecting sparse heterogeneous mixtures.
\newblock {\em Ann. Statist.}, 32(3):962--994.

\bibitem[Efron, 2009]{Efron2009}
Efron, B. (2009).
\newblock Correlated z -values and the accuracy of large-scale statistical
  estimates.
\newblock Preprint.

\bibitem[Efron et~al., 2001]{ETST2001}
Efron, B., Tibshirani, R., Storey, J.~D., and Tusher, V. (2001).
\newblock Empirical {B}ayes analysis of a microarray experiment.
\newblock {\em J. Amer. Statist. Assoc.}, 96(456):1151--1160.

\bibitem[Ferreira and Zwinderman, 2006]{FZ2006}
Ferreira, J.~A. and Zwinderman, A.~H. (2006).
\newblock On the {B}enjamini-{H}ochberg method.
\newblock {\em Ann. Statist.}, 34(4):1827--1849.

\bibitem[Finner et~al., 2009]{FDR2009}
Finner, H., Dickhaus, R., and Roters, M. (2009).
\newblock On the false discovery rate and an asymptotically optimal rejection
  curve.
\newblock {\em Ann. Statist.}, 37(2):596--618.

\bibitem[Finner et~al., 2007]{FDR2007}
Finner, H., Dickhaus, T., and Roters, M. (2007).
\newblock Dependency and false discovery rate: asymptotics.
\newblock {\em Ann. Statist.}, 35(4):1432--1455.

\bibitem[Finner and Roters, 2002]{FR2002}
Finner, H. and Roters, M. (2002).
\newblock Multiple hypotheses testing and expected number of type {I} errors.
\newblock {\em Ann. Statist.}, 30(1):220--238.

\bibitem[Gavrilov et~al., 2009]{GBS2009}
Gavrilov, Y., Benjamini, Y., and Sarkar, S.~K. (2009).
\newblock An adaptive step-down procedure with proven fdr control under
  independence.
\newblock {\em Ann. Statist.}, 37(2):619--629.

\bibitem[Genovese and Wasserman, 2004]{GW2004}
Genovese, C. and Wasserman, L. (2004).
\newblock A stochastic process approach to false discovery control.
\newblock {\em Ann. Statist.}, 32(3):1035--1061.

\bibitem[Glueck et~al., 2008]{Glueck2008}
Glueck, D., Mandel, J., Karimpour-Fard, A., Hunter, L., and Muller, K. (2008).
\newblock Exact calculations of average power for the benjamini-hochberg
  procedure.
\newblock {\em International Journal of Biostatistics}, 4(1):1103--1103.

\bibitem[Guo and Rao, 2008]{GR2008}
Guo, W. and Rao, M.~B. (2008).
\newblock On control of the false discovery rate under no assumption of
  dependency.
\newblock {\em Journal of Statistical Planning and Inference},
  138(10):3176--3188.

\bibitem[Lehmann and Romano, 2005]{LR2005}
Lehmann, E.~L. and Romano, J.~P. (2005).
\newblock Generalizations of the familywise error rate.
\newblock {\em Ann. Statist.}, 33:1138--1154.

\bibitem[Neuvial, 2008]{Neu2008}
Neuvial, P. (2008).
\newblock Asymptotic properties of false discovery rate controlling procedures
  under independence.
\newblock {\em Electron. J. Stat.}, 2:1065--1110.

\bibitem[Owen and Steck, 1962]{OS1962}
Owen, D.~B. and Steck, G.~P. (1962).
\newblock Moments of order statistics from the equicorrelated multivariate
  normal distribution.
\newblock {\em Ann. Math. Statist.}, 33:1286--1291.

\bibitem[Reiner-Benaim, 2007]{Reiner2007}
Reiner-Benaim, A. (2007).
\newblock F{DR} control by the {BH} procedure for two-sided correlated tests
  with implications to gene expression data analysis.
\newblock {\em Biom. J.}, 49(1):107--126.

\bibitem[Roquain, 2007]{Roq2007}
Roquain, E. (2007).
\newblock {\em Exceptional motifs in heterogeneous sequences. Contributions to
  theory and methodology of multiple testing}.
\newblock PhD thesis, Universit\'e Paris XI.

\bibitem[Roquain and van~de Wiel, 2009]{RW2009}
Roquain, E. and van~de Wiel, M. (2009).
\newblock Optimal weighting for false discovery rate control.
\newblock {\em Electron. J. Stat.}, 3:678--711.

\bibitem[Sarkar, 2002]{Sar2002}
Sarkar, S.~K. (2002).
\newblock Some results on false discovery rate in stepwise multiple testing
  procedures.
\newblock {\em Ann. Statist.}, 30(1):239--257.

\bibitem[Sarkar, 2008]{Sar2008}
Sarkar, S.~K. (2008).
\newblock On methods controlling the false discovery rate.
\newblock {\em Sankhya, Ser. A}, 70:135--168.

\bibitem[Shorack and Wellner, 1986]{SW1986}
Shorack, G.~R. and Wellner, J.~A. (1986).
\newblock {\em Empirical processes with applications to statistics}.
\newblock Wiley Series in Probability and Mathematical Statistics: Probability
  and Mathematical Statistics. John Wiley \& Sons Inc., New York.

\bibitem[Simes, 1986]{Sim1986}
Simes, R.~J. (1986).
\newblock An improved {B}onferroni procedure for multiple tests of
  significance.
\newblock {\em Biometrika}, 73(3):751--754.

\bibitem[Storey, 2003]{Storey2003}
Storey, J.~D. (2003).
\newblock The positive false discovery rate: a {B}ayesian interpretation and
  the {$q$}-value.
\newblock {\em Ann. Statist.}, 31(6):2013--2035.

\bibitem[Stuart, 1958]{Stuart1958}
Stuart, A. (1958).
\newblock Equally correlated variates and the multinormal integral.
\newblock {\em J. Roy. Statist. Soc. Ser. B}, 20:373--378.

\bibitem[Zeisel et~al., 2009]{Zei2009}
Zeisel, A., Zuk, O., and Domany, E. (2009).
\newblock Fdr control with adaptive procedures and fdr monotonicity.

\end{thebibliography}
\bibliographystyle{apalike} 

\medskip
\appendix
{\bf \Large Appendix}

\section{Useful lemmas}\label{notproof}

The following lemma is related to the proof of Theorem~2.1 in \cite{FZ2006} and to Lemma~8.1 (i) in \cite{RW2009}.

\begin{lemma} \label{lemma_stepup}
Consider a step-up procedure  $\SU(\mbf{t})$ using a given threshold $\mbf{t}$ testing $m$ null hypotheses with $p$-values $p_1$, ..., $p_m$ and rejecting
 $\hat{k}=|\SU(\mbf{t})|$ hypotheses. For a given $1\leq \l\leq m$, denote by $\hat{k}'_{(\l)}$ the number of rejections of the step-up procedure of threshold $(t_{j+\l})_{1\leq j \leq m-\l}$ over $m-\l$ hypotheses and using the $p$-values $p_{\l+1}$, ..., $p_m$. 
Then we have point-wise
$$\forall 1\leq i \leq \l,\: p_i\leq t_{\hat{k}}\:\: \Longleftrightarrow\:\: \forall 1\leq i \leq \l, \: p_i\leq t_{\hat{k}'_{(\l)}+\l}\:\: \Longleftrightarrow\:\:\hat{k}={\hat{k}'_{(\l)}+\l}.$$
\end{lemma}

\begin{proof}
First note that since $\hat{k}$ is nondecreasing in each coordinate of $(p_1,...,p_m)$, we always have $\hat{k}\leq {\hat{k}'_{(\l)}+\l}$.
Second, since $p_{(k)}\leq t_k$ is equivalent to $|\{1\leq j\leq m \telque p_j\leq t_k\}|\geq k$, the rejection number of $\SU(\mbf{t})$ can be defined as $\wh{k}= \max\{k\in\{0,1,...,m\}\telque |\{1\leq j\leq m \telque p_j\leq t_k\}|\geq k \}$. Hence,  $ \forall 1\leq i \leq \l, \: p_i\leq t_{\hat{k}'_{(\l)}+\l}$ is equivalent to $ |\{1\leq j\leq m \telque p_j\leq t_{\hat{k}'_{(\l)}+\l }\}| \geq \l +  |\{\l+1\leq j\leq m \telque p_j\leq t_{\hat{k}'_{(\l)}+\l } \}| $ which is equivalent to $ |\{1\leq j\leq m \telque p_j\leq t_{\hat{k}'_{(\l)}+\l }\}|  \geq \l + \hat{k}'_{(\l)}$, by definition of $ \hat{k}'_{(\l)}$. As a consequence, since $\hat{k}$ is a maximum and since $\hat{k}\leq {\hat{k}'_{(\l)}+\l}$, the latter is   equivalent to  $\hat{k}'_{(\l)}+\l = \hat{k}$. This establishes the second equivalence. The first equivalence easily comes from the second equivalence and using that $t_{\hat{k}}\leq t_{\hat{k}'_{(\l)}+\l}$ because $(t_k)_k$ is a nondecreasing sequence.
\end{proof}

For step-down procedures, we use the next lemma.
\begin{lemma} \label{lemma_down}
Consider a step-down procedure  $\SD(\mbf{t})$ using a given threshold $\mbf{t}$ testing $m$ null hypotheses with $p$-values $p_1$, ..., $p_m$ and rejecting
 $\tilde{k}=|\SD(\mbf{t})|$ hypotheses. Denote by $\tilde{k}_{(1)}$ (resp. $\tilde{k}'_{(1)}$) the number of rejections of the step-up procedure of threshold $(t_{j})_{1\leq j \leq m-1}$ (resp. $(t_{j+1})_{1\leq j \leq m-1}$) over $m-1$ hypotheses and using the $p$-values $p_{2}$, ..., $p_m$. 
Then we have point-wise
$$ p_1 \leq t_{\tilde{k}}\:\: \Longleftrightarrow\:\: p_1 \leq t_{\tilde{k}_{(1)}+1}\:\: \Longleftrightarrow\:\:\tilde{k}={\tilde{k}'_{(1)}+1}.$$
\end{lemma}
In the  above lemma, we underline that the assertion $p_1 \leq t_{\tilde{k}'_{(1)}+1} \Longrightarrow p_1 \leq t_{\tilde{k}} $ is not true in general.

\begin{proof}
Similarly to the step-up case, the rejection number of $\SD(\mbf{t})$ can be defined as $\tilde{k}= \max\{k\in\{0,1,...,m\}\telque \forall k'\leq k, \:|\{1\leq j\leq m \telque p_j\leq t_{k'}\}|\geq k' \}$. Also remark that we always have $\tilde{k}'_{(1)} +1 \geq  \tilde{k}$ and, by definition of $\tilde{k}$, for any $j$ we have $p_j\leq t_{\tilde{k}} \Leftrightarrow p_j\leq t_{\tilde{k}+1}$. First prove that $p_1 \leq t_{\tilde{k}}\:\:  \Leftrightarrow\:\:\tilde{k}={\tilde{k}'_{(1)}+1}$: using the definitions of $\tilde{k}$ and $\tilde{k}'_{(1)}$ we obtain $p_1 \leq t_{\tilde{k}} \Leftrightarrow |\{2\leq j \leq m \telque p_j\leq t_{\tilde{k}+1}\}|<\tilde{k} \Leftrightarrow \tilde{k}'_{(1)} < \tilde{k}  \Leftrightarrow \tilde{k}'_{(1)} +1= \tilde{k} $.
Second, we prove  $p_1 > t_{\tilde{k}} \Leftrightarrow p_1> t_{\tilde{k}_{(1)}+1}$: since we obviously have $\tilde{k}\geq \tilde{k}_{(1)}$, we get $p_1 > t_{\tilde{k}}   \Rightarrow p_1 > t_{\tilde{k}+1}   \Rightarrow p_1> t_{\tilde{k}_{(1)}+1}$. Conversely, if $p_1> t_{\tilde{k}_{(1)}+1}$, we get $|\{1\leq j \leq m\telque p_j \leq t_{\tilde{k}_{(1)}+1}\}| = |\{2\leq j \leq m\telque p_j \leq t_{\tilde{k}_{(1)}+1}\}|<\tilde{k}_{(1)}+1$ (by definition of $\tilde{k}_{(1)}$), hence $\tilde{k}_{(1)}+1> \tilde{k} $ (by definition of $\tilde{k}$), which implies  $\tilde{k}_{(1)} = \tilde{k} $, thus $p_1> t_{\tilde{k}+1}$ and finally $p_1> t_{\tilde{k}}$.
\end{proof}

\end{document}